\documentclass[12pt,pdftex]{amsart} 

\usepackage{0header}

\title[The Cartan-Hadamard conjecture for small volumes under Ricci bound]{A proof of the Cartan-Hadamard conjecture for small volumes under a Ricci curvature lower bound}

\author[M. Agnoletto]{M. Agnoletto}
\address{\textsc{Marcos Agnoletto}\newline
\indent CMCC, Federal University of ABC, Santo André, SP, Brazil}
\email{marcos.forte@ufabc.edu.br}

\author[M. F. Da Silva]{M. F. Da Silva}
\address{\textsc{M\'{a}rcio Fabiano Da Silva}\newline
\indent CMCC, Federal University of ABC, Santo André, SP, Brazil}
\email{marcio.silva@ufabc.edu.br.}

\author[S. Nardulli]{S. Nardulli}
\address{\textsc{Stefano Nardulli}\newline
\indent CMCC, Federal University of ABC, Santo André, SP, Brazil}
\email{stefano.nardulli@ufabc.edu.br}

\author[R. Resende]{R. Resende}
\address{\textsc{Reinaldo Resende}\newline
\indent Florida State University}
\email{rresende@fsu.edu}

 \date{}
\begin{document}

\date{}

\begin{abstract}
We prove the generalized Cartan--Hadamard conjecture, also known as the Aubin conjecture, in the small volume regime under a lower Ricci curvature bound, for any dimension $n\geq 2$. More precisely, we show that if $(M^n,g)$ is a Cartan-Hadamard manifold satisfying $\Sec_g\leq\bar k\leq 0$ and $\Ric_g\geq (n-1)\underline k$, then its isoperimetric profile is at least that of the model space of curvature $\bar k$ for small enough volumes.

We also reduce the proof of the generalized Cartan-Hadamard conjecture to establishing rigidity in the equality case.
\end{abstract}

\maketitle

\setcounter{tocdepth}{1}
{  \hypersetup{linkcolor=black}
  \tableofcontents
}

\section{Introduction}

A Cartan-Hadamard manifold $(M^n,g)$ is a complete, simply connected Riemannian manifold with nonpositive sectional curvature. The Cartan-Hadamard conjecture, also known as the Aubin conjecture, asserts that enclosing a given volume in a Cartan-Hadamard manifold requires at least as much perimeter as in the Euclidean space; see Theorem \ref{thm: main theorem}. Despite its extremely simple statement, the conjecture remains open in arbitrary dimensions. In dimensions $n=2, 3$, and $4$, the conjecture was settled by Weil \cite{Weil1926}, Kleiner \cite{kleiner1992isoperimetric}, and Croke \cite{Croke1984}, respectively. In very recent preprints \cite{chen2026isoperimetricinequalitycmchypersurfaces}, Chen, Ghomi, and Wang present a proof of the conjecture in dimension $3\leq n\leq 9$.

More generally, the \emph{generalized} Cartan-Hadamard conjecture states that, if $\Sec_g \leq \bar k \leq 0$, then enclosing a given volume in $(M^n,g)$ requires at least as much perimeter as in the model space of curvature $\bar k$. Concerning the \emph{generalized} Cartan--Hadamard conjecture, the case $n=2$ was settled by Bol \cite{Bol1941}, while Kleiner already proved the generalized conjecture in dimension $n=3$ \cite{kleiner1992isoperimetric}. The case $n\geq 4$ remains open. 

In this paper, we prove the generalized Cartan-Hadamard conjecture in every dimension $n\geq 2$ for small volumes. The only additional hypothesis we impose is a lower bound on the Ricci curvature. In this setting, our results are new even for the ``non-generalized'' version of the conjecture.

In order to state our main result, we let $(\mathbb{M}^n_b,\delta^b)$ denote the $n$-dimensional simply connected model space form of constant sectional curvature $b\in\mathbb{R}$ and we denote the isoperimetric profile of $(M^n, g)$ by $\ip{M^n}{g}$. We refer the reader to Section \ref{sec:Notation} (see \eqref{eq:IsoperimetricProfileDfefinition}) for the precise definition.

\begin{theorem}\label{thm: main theorem}
Let $(M^n, g)$ be a Cartan-Hadamard manifold with $\Sec_g \leq\bar k \leq 0$ and $\mathrm{Ric}_g \geq (n-1)\underline k $. Then there exists $v_0 = v_0(\underline k,\bar k, n)>0$ such that 
\begin{equation*}
    \ip{M^n}{g}(v)\geq \ipm{\bar k}(v)\text{ for any }v\in (0,v_0].
\end{equation*}
Moreover, if there exists a finite perimeter set $E\subset M$ with $\vol_g(E) = v\in (0,v_0]$ and $P_g(E) = \ipm{\bar k}(v)$, then $E$ is isometric to a geodesic ball in $\mathbb{M}^n_{\bar k}$.
\end{theorem}
\begin{remark}[Alexandrov spaces]\label{rmk:arbitrary-volumes}
In \Cref{thm: CH in alexandrov}, we actually prove Theorem \ref{thm: main theorem} in the more general context of non-collapsing boundaryless Alexandrov spaces with two-sided curvature bounds, cf. Sections \ref{sec:main reduction} and \ref{sec:small volumes}.
\end{remark}

In Section \ref{sec:arbitrary volumes}, we prove Theorem \ref{thm: assuming equality sem Ricci lower bound}, which gives a major reduction that states the following: if one assumes the rigidity in the equality case at volumes up to which the inequality in the conjecture is already known (under a lower curvature bound, namely Assumption \ref{assump:equality case}), then \Cref{thm: main theorem} holds for \emph{every volume} without any lower curvature bound. The assumption and precise statement are displayed below, see \Cref{ssec: alex} for the definitions regarding Alexandrov spaces.

\begin{assumption}\label{assump:equality case}
Let $(Z,d_Z)$ be an non-collapsing Alexandrov space with no Alexandrov boundary and curvature bounded from above by $\bar k\leq 0$ and from below by $\underline k$. Let $V>0$ be such that
\begin{equation*}
\ip{Z}{d_Z}(w)\geq \ipm{A}(w)\quad\text{for every } w\in(0,V].
\end{equation*}
We assume that for every finite perimeter set $E\subset Y$ with $\cH^n_{d_Z}(E) = V$ and $P_{d_Z}(E) = \ipm{A}(V)$, it holds that $(E,d_Y)$ is isometric to a geodesic ball in $\mathbb{M}^n_{\bar k}$.
\end{assumption}

Next, we state our result in the arbitrary volume regime which we prove in Section \ref{sec:arbitrary volumes}.

\begin{theorem}\label{thm: assuming equality sem Ricci lower bound}
Let $(M^n,g)$ be a Cartan-Hadamard manifold with $\Sec_g \leq \bar k \leq 0$. Assume in addition that Assumption \ref{assump:equality case} holds. Then we have
\begin{equation}\label{eqn: conjecture sem Ricci arbitrary vol}
    \ip{M^n}{g}(v)\geq \ipm{\bar k}(v)\text{ for any }v\geq 0.
\end{equation}
Moreover, if there exists a finite perimeter set $E\subset M$ with finite volume $v > 0$ and $P_g(E) = \ipm{\bar k}(v)$, then $E$ is isometric to a geodesic ball in $\mathbb{M}^n_{\bar k}$.
\end{theorem}

\begin{remark}\label{rmk: reduction Ricci}
Notice that Assumption \ref{assump:equality case} is needed at the level of Alexandrov spaces to establish Theorem~\ref{thm: assuming equality sem Ricci lower bound} for smooth Cartan-Hadamard manifolds. Furthermore, in Assumption \ref{assump:equality case}, the space also admits a \emph{lower} bound on the curvature. The latter fact comes from a reduction argument in Section \ref{sec: Ricci} showing that it suffices to prove the Cartan-Hadamard conjecture under a Ricci lower bound. We actually prove Theorem~\ref{thm: assuming equality sem Ricci lower bound} for Alexandrov spaces in Proposition~\ref{thm: assuming equality}.
\end{remark}

\begin{remark}\label{rmk: assumption critical volume}
By Theorem~\ref{thm: main theorem} for Alexandrov spaces (Theorem \ref{thm: CH in alexandrov}), the conclusion of Assumption \ref{assump:equality case} holds unconditionally when $V\leq v_0(v_{nc},a,A,n)$, so its content is the case $V>v_0$. 
\end{remark}

It's known that isoperimetric regions in Cartan-Hadamard manifolds need not be connected, by work of Hass \cite{Hass2016}. Establishing their boundedness is still an open problem that we settle under the assumptions of each theorem above. We prove this in Subsection \ref{ssec: bdd}.

\begin{corollary}\label{cor: isop region bdd}
Let $(M^n, g)$ be a Cartan-Hadamard manifold with $\Sec_g \leq\bar k \leq 0$ and $E\subset M^n$ be an isoperimetric region. If \eqref{eqn: conjecture sem Ricci arbitrary vol} holds true, then $E$ is bounded.
\end{corollary}
\begin{remark}
By Theorems~\ref{thm: main theorem} and \ref{thm: assuming equality sem Ricci lower bound}, if either $(M^n,g)$ has a Ricci lower bound and $\vol_g(E)\leq v_0$, or Assumption~\ref{assump:equality case} holds, the corollary applies.
\end{remark}

\subsection{Brief history of the conjecture.} The Cartan-Hadamard conjecture dates back to 1926, when Weil settled the case $n=2$ \cite{Weil1926,Berger2003}; this was later reproved by Beckenbach-Radó in 1933 \cite{BeckenbachRado1933}. Aubin formulated the conjecture in dimensions $n\geq 3$ in 1975 \cite{Aubin1975}, and related problems were subsequently posed by Gromov \cite{Gromov1981,Gromov1999} and Burago-Zalgaller \cite{BuragoZalgaller1980,BuragoZalgaller1988}. The case $n=3$ was later settled by Kleiner \cite{kleiner1992isoperimetric}, while Croke \cite{Croke1984}, using a completely different method, proved that an Euclidean type inequality holds, i.e., $I_{(M^n, g)}\ge \gamma_n I_{(\mathbb{R}^n, \delta_{euc})}$ with the constant $\gamma_n\in (0, 1]$. Croke also proved that $\gamma_n=1$ if, and only if, $n=4$.  In this way, he proved the Cartan-Hadamard conjecture only in dimension $n=4$. Alternative proofs are available in \cite{RitoreSinestrari2010,Schulze2008} for $n=3$ and in \cite{kloeckner2019cartan} for $n=4$. It is somewhat surprising that Croke's proof in dimension \(n=4\) appeared in the literature before Kleiner's proof in dimension \(n=3\), and that the \emph{non-optimal} inequality $\ip{M^3}{g} \geq \gamma_3 I_{(\mathbb{R}^3, \delta_{euc})}$ is due to Croke, whereas the \emph{sharp} inequality $\ip{M^3}{g}\geq I_{(\mathbb{R}^3, \delta_{euc})}$ was later established by Kleiner.

We highlight the great results and very detailed discussions about the conjecture by Ghomi and Spruck in \cite{GhomiSpruck2022} and its generalized version in $n=4$ by Kloeckner and Kuperberg in \cite{kloeckner2019cartan}, as well as the following resources \cite{druet2010isoperimetric,ritore2023isoperimetric}. The conjecture and related problems have also been studied from several other perspectives; we refer to the nonexhaustive list \cite{Wheeler2026,ghomi2026total,ghomi2026isoperimetric,ghomi2026local,Druet2002,MorganJohnson2000,Schulze2020,cao2003curvature,croke1980some,ghomi2023minkowski,osserman1978isoperimetric}. In a different direction, Hass showed that isoperimetric regions in Cartan-Hadamard manifolds need not be connected \cite{Hass2016}.

The known proofs in dimensions three and four rely on rather different mechanisms. Croke's argument is specific to dimension four and only gives the \emph{sharp} inequality in this dimension. Kleiner's proof proceeds instead through the total curvature conjecture: he first establishes the corresponding total curvature estimate and then derives the Cartan-Hadamard conjecture as a consequence. A central ingredient in his argument is a nontrivial use of the Gauss-Bonnet theorem, and the method does not extend simply by replacing it with higher-dimensional versions of Gauss-Bonnet. More recently, Ghomi and Spruck \cite{GhomiSpruck2022} proved that the implication from the total curvature conjecture to the Cartan-Hadamard conjecture holds in every dimension $n\geq 2$. However, the validity of the total curvature conjecture remains an open problem in arbitrary dimensions as well.

\begin{remark}
Less than a week before we posted this paper, Chen, Ghomi, and Wang~\cite{chen2026cartanhadamardconjecturedimension} presented a proof of the conjecture in dimension $n=5$. Seven days later, the same authors~\cite{chen2026isoperimetricinequalitycmchypersurfaces} put forward a proof of the conjecture in dimensions $3\leq n \leq 9$. The authors explicitly acknowledge the use of AI assistance in both papers.

We are not yet familiar with their entire proofs since they're extremely recent. However, we can state that our approach is completely different than theirs.
\end{remark}

\subsection{Brief history of the generalized conjecture.} The \emph{generalized} Cartan-Hadamard conjecture was settled in dimension $n=2$ by Bol in 1941 \cite{Bol1941}, and in dimension $n=3$ by Kleiner in 1992 \cite{kleiner1992isoperimetric}. An alternative proof for $n=3$ using mean curvature flow follows from the more general result of Schulze \cite{Schulze2020}.

The case $n=4$ is more delicate. Kloeckner and Kuperberg \cite{kloeckner2019cartan} extended Croke's method, under extra structural assumptions on the ambient manifold and for small volumes, proved the generalized conjecture for $n=4$. Their argument is also quite different as it uses optical transport and linear programming.

Druet \cite{Druet2002} proved a local version of the generalized conjecture near any point where the scalar curvature is \emph{strictly} less than $n(n-1)\bar k$.

\subsection{What is new?}

The approach developed in this paper is entirely different from the previous methods used to approach the conjecture. In particular, we neither assume nor prove the total curvature conjecture, and our argument does not pass through it whatsoever. Consequently, even in dimensions where the Cartan-Hadamard conjecture is already known, our result yields a \emph{genuinely different proof}, under the additional assumption of a lower Ricci curvature bound. \emph{Our proof in contrast with the known results in the literature is not dimension dependent}. Moreover, none of the previous approaches to the conjecture form part of the mechanism of our proof. 

The lower Ricci curvature bound, and the reduction to this case, cf. Section \ref{sec: Ricci}, play a fundamental role in our argument through the generalized existence and compactness theory for isoperimetric regions developed by the third named author in \cite{nardulli2014generalized,MunozFloresnardulli2020}. Without such a lower bound, the corresponding existence and compactness theory is known to fail. Indeed, one of the main difficulties in approaching the Aubin conjecture is that minimizing sequences—namely, sequences of sets of volume $v$ whose perimeters converge to $\ip{M^n}{g}(v)$—may run off to infinity. This phenomenon can already be seen in the hyperbolic paraboloid $(P,h)$ in $\R^3$, which satisfies $\Sec_h\leq 0$ and $\mathrm{Ric}_h \geq -1$, as it does not contain any isoperimetric region. We show how the lower Ricci curvature bound allows us to overcome this lack of genuine compactness using a generalized one.

\subsection{Overview of the strategy of the proof}

The small-volume case, Theorem~\ref{thm: main theorem}, is proved through two major reductions: first, to generalized isoperimetric regions; and second, to regions represented as normal graphs. More precisely:
\begin{enumerate}
\item In Section~\ref{sec:main reduction}, we use the lower Ricci curvature bound to invoke the generalized existence and compactness theory described above. This reduces the proof of Theorem~\ref{thm: main theorem} to Proposition~\ref{prop: reduction assuming existence}, cf. Section~\ref{sec:main reduction}. In other words, the lower curvature bound allows us to apply the generalized existence which reduces the proof of Theorem~\ref{thm: main theorem} to proving it at the limit spaces that contain an isoperimetric region. On the other hand, the generalized existence forces us to work in the setting of Alexandrov metric spaces.

\item Once this reduction has been established, we prove Proposition~\ref{prop: reduction assuming existence} in Section~\ref{sec:small volumes}. The proof centers on showing that generalized isoperimetric regions of small volumes are represented by normal graphs over geodesic spheres. It then suffices to prove the conjecture for normal graphs over geodesic spheres, which we do by hand in Lemma~\ref{lem: isoperimetric inequality for normal graphs}.
\end{enumerate}

The proof of Theorem~\ref{thm: assuming equality sem Ricci lower bound} requires an additional step because no Ricci lower bound is assumed in that result. Specifically, in Section \ref{sec: Ricci}, we show that it suffices to prove Theorem~\ref{thm: assuming equality sem Ricci lower bound} under the additional assumption $\mathrm{Ric}_g \geq (n-1)\underline{k}.$ Since the third author and Mu\~noz-Flores show in \cite{NardulliFlores2020} that the isoperimetric profile \eqref{eq:IsoperimetricProfileDfefinition} can be defined by minimizing over bounded smooth sets, we use this boundedness to modify the metric $g$ outside a sufficiently large ball, constructing a metric $\hat g$ with the desired Ricci curvature lower bound. The fact that the metric $g$ and $\hat g$ coincide inside that very large ball then completes the reduction after a simple computation.

Once the reduction to the Ricci lower-bound setting is complete, we essentially follow the strategy outlined in steps (1) and (2) above.

\subsection{AI disclosure} Publicly available LLMs were used to improve the grammar and flow of the introduction and to assist with literature searches. 

\subsection{Acknowledgments} 

M. Agnoletto acknowledges financial support from CAPES Grant No.~88887.667684/2022-00, CNPq Grants No.~201543/2024-9 and 201742/2025-0, and FAPESP Grant No.~2021/05256-0.

S. Nardulli was supported by FAPESP Young Investigator Grant No.~2021/05256-0, CNPq Research Productivity Fellowship A No.~309345/2026-0 and 1D No.~12327/2021-8, and CNPq Grants No.~441922/2023-6, 409669/2025-3, and No.~200817/2025-6. 

R. Resende gratefully acknowledges support from FAPESP Grant No.~2021/05256-0 and CNPq Grants No.~409669/2025-3 and ~441922/2023-6 for research visits to UFABC during which part of this work was carried out.

\section{Notation}\label{sec:Notation}

For each $b \in \R$, let $\mathbb{M}^n_b$ denote the $n$-dimensional simply connected model space of constant curvature $b$. We let $
    \delta_{euc}:=\delta^0$ and $\omega_{n-1}:=\vol_{\sigma}\bigl(\partial B_1^n(0)\bigr)$ where $\sigma$ is the round metric on the sphere.

Given a metric space $(X,d)$, we denote by $\cH^s_d$ the $s$-dimensional Hausdorff measure induced by $d$. For a set $A\subset X$, its diameter is $\diam_d(A):=\sup\{d(x,y):x,y\in A\}$. The symmetric difference of two sets $E,F\subset X$ is denoted by $E\mathbin{\Delta}F:=(E\setminus F)\cup(F\setminus E)$, where for any given sets $A,B$ we define $A\setminus B$ as the set theoretic difference of $A$ and $B$, i.e., $A\setminus B:=\left\{a\in A:a\notin B\right\}$.

For a Riemannian manifold $(M,g)$, we denote by $\vol_g$ its Riemannian volume measure and denote the sectional curvature of $g$ by $\Sec_g$. Given $k_1,k_2\in\R$, the notation $k_1\leq\Sec_g\leq k_2$
means that $k_1\leq\Sec_g(\sigma)\leq k_2$ for every two-dimensional subspace $\sigma\subset T_xM$ and every $x\in M$. {For $x \in M$ and $R > 0$, we denote by $B_R^g(x)$ the geodesic ball with respect to the metric $g$ centered at $x$ of radius $R$.}

Let $(M,g)$ be a Riemannian manifold and $\Sigma\subset M$ be a smooth embedded hypersurface with trivial normal bundle, and let
$\nu$ be a choice of unit normal vector field along $\Sigma$. Given a function
$u:\Sigma\to\mathbb{R}$ such that $u(x)\nu(x)$ lies in the domain of the
normal exponential map for every $x\in\Sigma$, the normal graph of $u$
over $\Sigma$ is defined by
\[
\operatorname{graph}_{\Sigma}^{g}(u)
:=
\left\{
\exp_{x}^{g}\bigl(u(x)\nu(x)\bigr):x\in\Sigma
\right\},
\]
where $\exp^{g}$ denotes the Riemannian exponential map.

Let $(X,d,\mu)$ be a metric measure space. A measurable set $E\subset X$ is a set of finite perimeter if $\ind E\in BV(X)$. Its perimeter in an open set $U\subset X$ is denoted by $P_d(E;U):=|D\ind E|(U)$, and we set $P_d(E):=P_d(E;X)$. In a Riemannian manifold $(M,g)$, we use the notations $P_g(E;U):=P_{d_g}(E;U)$ and $P_g(E):=P_{d_g}(E)$, where $d_g$ is the canonical path length metric on $M$ induced by the Riemannian metric tensor $g$. 

A sequence of measurable sets $\{E_j\}_j$ converges to $E$ in $L^1_{\mathrm{loc}}(X,\mu)$ if $\mu\bigl((E_j\mathbin{\Delta}E)\cap K\bigr)\longrightarrow 0$ for every compact set $K\subset X$. It converges to $E$ in $L^1(X,\mu)$ if
$\mu(E_j\mathbin{\Delta}E)\longrightarrow 0$.

In view of \cite[Definition~1.2 and Theorem~1]{NardulliFlores2020}, the isoperimetric profile of a complete Riemannian manifold can equivalently be defined by minimizing over sets of finite perimeter. Accordingly, we use the following formulation. 

The {\emph{isoperimetric profile of a Riemannian manifold $(M,g)$}} is the function
\begin{equation}\label{eq:IsoperimetricProfileDfefinition}
    \ip{M}{g}(v)
    :=
    \inf\bigl\{
        P_g(E):
        E\subset M\text{ is a set of finite perimeter and }
        \vol_g(E)=v
    \bigr\},
\end{equation}
defined for any $v\in(0,\vol_g(M))$. A finite perimeter set $E\subset M$ realizing this infimum is called an isoperimetric region of volume $v$. 

A {\emph{minimizing sequence for $\ip{M}{g}(v)$}} is a sequence of sets of finite perimeter $\{E_j\}_j$ such that
\begin{equation*}
    \vol_g(E_j)=v
    \quad\mbox{for every }j,
    \qquad
    P_g(E_j)\longrightarrow\ip{M}{g}(v).
\end{equation*}

For the following notions, we refer the reader to \cite[Section~1.3]{AmbrosioBrueSemola2019}. We write
\begin{equation*}
    (X_j,d_j,\mu_j,x_j)
    \xrightarrow{\mathrm{pmGH}}
    (X,d,\mu,x)
\end{equation*}
to denote pointed measured Gromov-Hausdorff convergence. More precisely, this means that the spaces can be isometrically embedded into a common complete metric space $(Z,d_Z)$ in such a way that the base points converge, bounded balls converge in the Hausdorff sense, and the pushforwards of the measures converge weakly against functions in $C_c(Z)$, cf.\ \cite[Definition~2.12]{antonelli2022isoperimetric}.

Under such a realization, we say that measurable sets $E_j\subset X_j$ {\emph{converge $L^1$-strongly to a measurable set $E\subset X$}} if
\begin{equation*}
    (\iota_j)_\#(\ind{E_j}\mu_j)
    \rightharpoonup
    \iota_\#(\ind E\mu)
\end{equation*}
weakly as measures on $Z$ and
\begin{equation*}
    \mu_j(E_j)\longrightarrow\mu(E),
\end{equation*}
where $\iota_j:X_j\to Z$ and $\iota:X\to Z$ are the corresponding isometric embeddings, cf.\ \cite[Definition~2.15]{antonelli2022isoperimetric}.

\subsection{Alexandrov spaces.}\label{ssec: alex} We define and recall some basic properties of these spaces, we refer the reader to \cite{alexander2019invitation,kapovitch2020cd} for the notions of Alexandrov curvatures and definitions of metric classes as $\textup{CAT}$ and $\textup{CBB}$, and to \cite{Gigli2018,AGS2014Calculus} for the theory of $\textup{RCD}$ spaces. 

\begin{definition}\label{def:alex}
Let $\bar k\geq \underline k$ be real numbers and $(X, d)$ a complete and locally compact metric space. We say that:

\begin{enumerate}[\upshape (i)]
    \item $(X,d)$ is an Alexandrov space with curvature bounded from above by $\bar k$ if it is $\textup{CAT}(\bar k)$ and bounded from below by $\underline k$ if it is $\textup{CBB}(\underline k)$;

    \item $(X,d,\cH^n_d)$ is non-collapsing if there exists a real number $v_{nc}$ such that$$\inf_{x\in X} \cH^n_d(B_1^d(x)) \geq v_{nc} > 0;$$

    \item we refer the reader to \cite[\S 10.10]{BuragoBuragoIvanov2001} for the definition of Alexandrov boundary, the specific definition is not relevant to this article as we only use it as an assumption to apply results in the literature.
\end{enumerate}
\end{definition}

In the following proposition, we collect groundbreaking results obtained for Alexandrov spaces that will be vital in this article. 

\begin{proposition}[Properties of Alexandrov spaces]\label{rmk:alex}
Let $(X,d)$ be an Alexandrov space with curvature bounded from above by $\bar k \leq 0$ and bounded from below by $\underline k$ with no Alexandrov boundary. Then
\begin{enumerate}[\upshape (i)]
    \item\label{rmk: Alex simply connected} $(X,d)$ is simply-connected;

    \item\label{rmk: Alex RCD} $(X,d,\cH^n_d)$ is a $\textup{RCD}((n-1)\underline{k}, n)$ space;
        
    \item\label{rmk: Alex C1alpha} $(X,d)$ admits a $C^{3,\alpha}$ atlas and a $C^{1,\alpha}$ metric $g$, for any $\alpha \in (0,1)$, whose induced distance on $X$ coincides with $d$. Moreover, the metric $g$ is actually of class $W^{2,p}$ for every $p\geq 1$;

    \item\label{rmk: Alex approximants} there exists a sequence of Riemannian metrics $g_i\in C^\infty$ on $X$ that converge to $g$ in $C^{1,\alpha}$-topology for any $\alpha\in (0,1)$ or in $W^{2,p}$-topology for any $p\geq 1$. Furthermore, let $d_i$ be the distance induced by $g_i$, then each $(X, d_i)$ is $\textup{CAT}(\bar k_i)$ and $\textup{CBB}(\underline k_i)$ with     \begin{equation*} \underline k \leq \liminf_i \underline k_i \leq \limsup_i \bar k_i \leq \bar k.\end{equation*}
    {Moreover, $\bar k_i$ can be taken strictly smaller than $\bar k$.}
\end{enumerate}
\end{proposition}
\begin{remark}
In view of \cite{Druet2002}, we observe the following. Let $(M^n,g)$ be a Cartan-Hadamard manifold, and $g_i$ given by \eqref{rmk: Alex approximants}, which satisfy $\Sec_{g_i}<0$. By \cite[Theorem 1]{Druet2002}, each $x\in M$ admits a Druet radius $r_i(x)>0$ such that $P_{g_i}(E)>\ipeuc(\vol_{g_i}(E))$ for every set $E\subset B_{r_i(x)}^{g_i}(x)$ of finite and positive volume. If $r(x):=\liminf_i r_i(x)>0$, continuity allows us to pass to the limit and obtain $P_g(E)\geq\ipeuc(\vol_g(E))$ for every such set $E\Subset B_{r(x)}^g(x)$.
However, the available estimates ensuring a uniform positive lower bound for the Druet radii require a uniform lower Ricci curvature bound in general. This brings us back to the setting of \Cref{thm: main theorem}, which is more general because it only requires the volume to be small.

Moreover, the Druet radius admits a lower bound in terms of the harmonic radius. By \cite[Proposition 5.20]{Pigola2024}, the harmonic radius is infinite only if the manifold is isometric to the Euclidean space.
\end{remark}
\begin{proof}
Item \eqref{rmk: Alex simply connected} follows since Alexander, Kapovich, and Petrunin, in \cite[Corollary  2.2.6]{alexander2019invitation}  showed that $(X,d)$ is contractible. Item \eqref{rmk: Alex RCD} follows from work of Petrunin in \cite[Main theorem]{Petrunin2011} which ensures that it satisfies the $\mathrm{CD}((n-1)\underline k,n)$ and then we can apply a result by Kapovich and Ketterer in \cite[Theorem 1]{kapovitch2020cd}. Item \eqref{rmk: Alex C1alpha} is a famous theorem proved by Berestovskij and Nikolaev in \cite[Theorem 14.1]{berestovskij1993multidimensional}. Item \eqref{rmk: Alex approximants} is stated in \cite[Theorem 15.1, Remark 1]{berestovskij1993multidimensional} by Berestovskij and Nikolaev except for the moreover part. 

We are then left to prove the moreover part of item \eqref{rmk: Alex approximants}. This is a routine argument. 

First, assume that $\bar k<0$. By the approximation theorem cited above \cite[Theorem 15.1, Remark 1]{berestovskij1993multidimensional}, choose complete smooth metrics $h_j\to g$ with $\underline k-\delta_j\leq\Sec_{h_j}\leq\bar k+\delta_j$, where $\delta_j\downarrow0$ and $\delta_j<-\bar k$. Define $\lambda_j:=\frac{\bar k+\delta_j}{\bar k-\delta_j}\in(0,1)$ and $g_j:=\lambda_jh_j$. Since sectional curvature scales by the inverse of the metric scaling factor, we obtain
$\frac{\underline k-\delta_j}{\lambda_j}\leq\Sec_{g_j}\leq\bar k-\delta_j<\bar k$. As $\lambda_j\to1$, the metrics $g_j$ are complete and have the required convergences, and the lower and upper curvature bounds tend to $\underline k$ and $\bar k$, respectively.

We now assume that $\bar k=0$.

Fix $o\in X$ and put $r=d(o,\cdot)$, $\rho=1+r^2$. We have that $\rho\in C^{1,1}_{\mathrm{loc}}(X)$ and $|D\rho|_g=2r$ by \cite[Proposition~4.1 and~1.6]{KL}. By adapting the argument in \cite[Section~2.4]{KL} using the \emph{global} condition of $X$ being $\textup{CAT}(0)$ and \cite[Theorem~9.25]{AKP} using the \emph{global} condition $\textup{CBB}(\underline k)$, we obtain the following global inequality almost everywhere
\begin{equation}\label{eq:rho}
  2g\le\operatorname{Hess}_g\rho\le2(1+\sqrt{-\underline k} r)\,g .
\end{equation}

Fix a locally finite atlas $(U_\nu)_{\nu\ge1}$ and a partition of unity
$(\chi_\nu)$ with $K_\nu=\operatorname{supp}\chi_\nu\subset U_\nu$ compact. In
$U_\nu$, $\operatorname{Hess}_g(\rho*\eta_\sigma)-(\operatorname{Hess}_g\rho)*\eta_\sigma\to0$ uniformly on
$K_\nu$ for a mollifier $\eta_\sigma\ge0$, because the Christoffel symbols $\Gamma$ and $D\rho$ are
continuous. Positivity of $\eta_\sigma$ therefore preserves~\eqref{eq:rho} up to
small errors. Choose $\rho_\nu:=\rho*\eta_{\sigma_\nu}$ with value, gradient and Hessian errors
at most $2^{-\nu}\big/\bigl(8+8\sup_{K_\nu}(|\operatorname{Hess}_g\chi_\nu|_g+2| D\chi_\nu|_g)\bigr)$
on $K_\nu$, and put $$f=\sum_\nu\chi_\nu\rho_\nu.$$
Since $\sum D\chi_\nu=0=\sum\operatorname{Hess}_g\chi_\nu$, then $\operatorname{Hess}_gf = \sum_\nu \chi_\nu \operatorname{Hess}_g\rho_\nu + Err$ where $Err$ only involves term with $\rho_\nu-\rho$ and $D(\rho_\nu-\rho)$. Hence $|f-\rho|\le\frac12$, $|D f- D\rho|_g\le1$ and, by \eqref{eq:rho}, $g\le\operatorname{Hess}_gf\le(3+2\sqrt{-\underline k} r)g$. As $f\ge\frac12+r^2$, for constants $A_0,B_0>0$ depending only on $\underline k$, we obtain
\begin{equation*}
  g\le\operatorname{Hess}_gf\le A_0\sqrt f\,g \text{ and } |D f|_g^2\le B_0f.
\end{equation*}

Let $h_j$ be the metric given by item \eqref{rmk: Alex approximants}, $A:= 3A_0$, and $B := 2B_0$, by the $W^{2,p}$ convergence of $h_j$ to $g$ and the estimates above, we readily obtain
\begin{gather}
  \underline k - \varepsilon_j\le\Sec_{h_j}\le\varepsilon_j,\qquad \tfrac12g\le h_j\le\tfrac32g,\label{eq:h1}\\
  \tfrac12 h_j\le\operatorname{Hess}_{h_j}f\le A\sqrt f\,h_j,\qquad |D f|_{h_j}^2\le Bf.\label{eq:h2}
\end{gather}

Let $g_j=e^{4\varepsilon_jf}h_j$, and let $P$ be a $2$-plane with
$h_j$-orthonormal basis $e_1,e_2$. By a direct computation, we get
\begin{equation*}
  e^{4\varepsilon_jf}\Sec_{g_j}(P)=\Sec_{h_j}(P)
  -2\varepsilon_j\sum_{\ell=1}^2\operatorname{Hess}_{h_j}f(e_\ell,e_\ell)
  -4\varepsilon_j^2\bigl|\operatorname{proj}_{P^\perp}\nabla^{h_j}f\bigr|_{h_j}^2 .
\end{equation*}

\emph{Upper bound.} By \eqref{eq:h1} and \eqref{eq:h2}, the right-hand side is at
most $\varepsilon_j-2\varepsilon_j$, so
$\Sec_{g_j}\le-\varepsilon_je^{-4\varepsilon_jf}<0$. Thus the lower Hessian
bound absorbs the possibly positive curvature error of $h_j$.

\emph{Lower bound.} By \eqref{eq:h1} and \eqref{eq:h2}, the right-hand side is at
least $\underline k -\varepsilon_j-4A\varepsilon_j\sqrt f-4B\varepsilon_j^2f$. We have
$\underline k -\varepsilon_j\le0$, $f\ge0$, $\sqrt{\varepsilon_j f}\,e^{-4\varepsilon_j f}\le(8e)^{-1/2}$, and $\varepsilon_j f\,e^{-4\varepsilon_j f}\le(4e)^{-1}$, this implies
\begin{equation*}
    \Sec_{g_j}\ge \underline k -\frac{\sqrt2A}{\sqrt e}\sqrt{\varepsilon_j}-\Bigl(1+\frac Be\Bigr)\varepsilon_j
  \ge  \underline k  -C\sqrt{\varepsilon_j}.
\end{equation*}

This shows that we can choose $\bar k_j < 0$ as desired. The corresponding convergences follow trivially by the construction.
\end{proof}

\subsection{Limits of Alexandrov spaces}

We recall a couple of definitions about the limiting behavior of spaces with lower Ricci curvature bounds and in particular of Alexandrov spaces.

\begin{definition}
Let $\alpha \in (0,1)$. A Riemannian manifold $(M,g)$ has $C^{1,\alpha}$-locally asymptotic bounded geometry if it has Ricci curvature bounded below, it is non-collapsing, and if for every diverging sequence of points $(p_j)$, there exist a subsequence $(p_{j_l})$ and a pointed  manifold $(M_{\infty},g_{\infty},p_{\infty})$, with $g_{\infty}$ of class $C^{1,\alpha}$, such that $(M,g,p_{j_l}) \to (M_{\infty},g_{\infty},p_{\infty})$ in the pointed $C^{1,\alpha}$-topology
\end{definition}

Later we use the fact that manifolds with two-sided bounds on the sectional curvature admit particularly good properties for their pointed limits, as we remark below.

\begin{remark}\label{rmk:locally asymptotic bdd geometry}
If $(Z,g_Z)$ is a Cartan-Hadamard manifold with $\mathrm{Ric}_{g_Z} \geq (n-1)\underline k$, by \cite[Theorem 1.1 and Remark 2.4(i)]{Anderson1990}, we have that $(Z,g_Z)$ has $C^{1,\alpha}$-locally asymptotic bounded geometry.
\end{remark}

The next is the well-known smooth version of the pointed Gromov-Hausdorff convergence which will be used below, cf. \cite[Section 11.3.2, pg. 414]{Petersen2016}.

{\begin{definition}\label{def:flat convergence}
Let $\alpha\in(0,1)$,
$\{(M_i,g_i,p_i)\}_{i\in\mathbb{N}}$ be a sequence of pointed smooth
Riemannian manifolds, and $E_i\subset M_i$ be locally finite
perimeter sets for all $i \in \mathbb{N}$. Suppose that
$(M_i,g_i,p_i)\to(M_\infty,g_\infty,p_\infty)$ in the pointed
$C^{1,\alpha}$-flat topology, and let $E_\infty\subset M_\infty$ be a
locally finite perimeter set.

We say that $(E_i,g_i,p_i)$ converges to
$(E_\infty,g_\infty,p_\infty)$ in the pointed
$C^{1,\alpha}$-flat topology if, for every $R>0$, there exist a domain $\Omega_R$ such that $B_R^{g_\infty}(p_{\infty}) \subset \Omega_R \subset M_{\infty}$, an
index $i_R\in\mathbb{N}$, and $C^{2,\alpha}$ embeddings
$\Psi_{i,R}\colon \Omega_R\to M_i$ such that for all $i\geq i_R$ we have that
$B_R^{g_i}(p_i)\subset\Psi_{i,R}(\Omega_R)$,
$\Psi_{i,R}(p_\infty)=p_i$,
$\Psi_{i,R}^{*}g_i\to g_\infty$ in $C^{1,\alpha}(\Omega_R)$, and{
\[
\vol_{g_\infty}\Bigl(B_R^{g_\infty}(p_\infty)\cap\bigl(\Psi_{i,R}^{-1}(E_i)\mathbin{\Delta}E_\infty\bigr)\Bigr)\longrightarrow 0.
\]}
\end{definition}}

\section{Preliminary results}

In \cite[Theorem~3.4]{MorganJohnson2000}, Morgan and Johnson prove an isoperimetric inequality for compact manifolds; \cite[Section 4.6]{MorganJohnson2000} extends this to noncompact manifold with compact quotient by the isometry group. In \cite[Prop.~3.2]{MondinoNardulli2016}, Mondino and Nardulli prove a similar result for noncompact $M^n$ satisfying a Ricci curvature lower bound. Both results apply to space forms and readily give the next lemma. However, since the proof is substantially simpler in the case of space forms, we present it here.

\begin{lemma}\label{lem: Aubin in model spaces}
If $a \leq A \leq 0$, then $\ipm{a} \geq \ipm{A}$.
\end{lemma}
\begin{remark} As it is easily seen from the proof, when $a<A$, we get $\ipm{a}>\ipm{A}$.
    \end{remark}
\begin{proof}
For each $b \leq 0$, by \cite[Theorem 4.35 and Theorem 4.39]{ritore2023isoperimetric}, the isoperimetric regions in $\mathbb{M}^n_b$ are geodesic balls. Therefore, for all $v \geq 0$, $\ipm{b}(v) = P_b(r_b), $ where $P_b(r)$ denote the perimeter of a geodesic ball of radius $r$ in $\mathbb{M}^n_b$ and $r_b$ is such that the volume of the geodesic ball of radius $r_b$ in $\mathbb{M}^n_b$ is equal to $v$. For $b \leq 0$, define
\begin{displaymath}
s_b(r) :=
\begin{cases}
r, & b = 0,\\
\dfrac{\sinh\left(\sqrt{-b} \ r\right)}{\sqrt{-b}}, & b < 0.
\end{cases}
\end{displaymath}

Then, for all $r \geq 0$,
\begin{displaymath}
    P_b(r) = \omega_{n-1}\left(s_b(r)\right)^{n-1}\text{ and }
V_b(r) = \omega_{n-1}\int_0^r\left(s_b (t)\right)^{n-1}dt, 
\end{displaymath}

Since $V_b'(r) > 0$ for all $r > 0$, the function $V_b$ is invertible. Using the inverse function theorem, we obtain
\begin{displaymath}
\ipm{b}'(v) = P_b'(R_b(v))R_b'(v) = \dfrac{P_b'(R_b(v))}{V_b'(R_b(v))},
\end{displaymath}

for all $v \geq 0$, where $R_b(v) := V_b^{-1}(v)$. Using that $V_b^\prime(r) = P_b(r)$ for any $r>0$ and the last displayed equation, we straightforwardly see that for $b < 0$
\begin{equation}\label{eqn: isop profile and s_b}
\ipm{b}'(v) = (n-1)\frac{s_b^\prime(R_b(v))}{s_b(R_b(v))} = (n-1)\sqrt{-b}\coth(\sqrt{-b} \ R_b(v)),
\end{equation}
and, for $b=0$
\begin{equation}\label{eqn: isop profile and s_0}
\ipm{0}'(v) = (n-1)\frac{s_0^\prime(R_0(v))}{s_0(R_0(v))} = (n-1) \dfrac{1}{R_0(v)}.
\end{equation}

Since the conclusion is immediate when $a=A$, it only remains to treat two cases. Firstly, suppose that $a<A<0$. Since the function in the right-hand side of \eqref{eqn: isop profile and s_b} is nondecreasing with respect to $b$ and $a < A$, we derive that 
\begin{equation}\label{eqn: comparison s_b fnc}
\ipm{a}'(v) \geq (n-1)\frac{s_A^\prime(R_a(v))}{s_A(R_a(v))}.
\end{equation}

Moreover, by the Bishop–Gromov Volume Comparison Theorem \cite[Theorem 11.19]{Lee2019}, we have that $R_a(v) \leq R_A(v)$. Combining this with the fact that the right-hand side of \eqref{eqn: comparison s_b fnc} is nonincreasing in the radius, we deduce that, for all $v \geq 0$,
\begin{equation*}
\ipm{a}'(v) \geq (n-1)\frac{s_A^\prime(R_A(v))}{s_A(R_A(v))} \overset{\eqref{eqn: isop profile and s_b}}{=} \ipm{A}'(v).
\end{equation*}

Integrating this inequality and using that both isoperimetric profiles vanish at $v=0$, we conclude this case. Secondly, suppose that $a < A = 0$. Since for every $r>0$ and every $b<0$ we have
\begin{displaymath}
(n-1)\sqrt{-b}\coth(\sqrt{-b}r) \geq (n-1)\dfrac{1}{r},
\end{displaymath}
it follows from \eqref{eqn: isop profile and s_b} that
\begin{equation*}
\ipm{a}'(v) \geq (n-1)\frac{1}{R_a(v)}.
\end{equation*}

Combining the last displayed equation with the inequality $R_a(v)\leq R_0(v)$, we obtain that
\begin{displaymath}
\ipm{a}'(v) \geq (n-1)\dfrac{1}{R_0(v)} \overset{\eqref{eqn: isop profile and s_0}}{=} \ipm{0}'(v).
\end{displaymath}

Integrating this inequality and using that both isoperimetric profiles vanish at $v=0$, we conclude the proof.
\end{proof}

We will use some properties of the isoperimetric profile of model spaces that we prove below.

\begin{lemma}\label{lem: isop profile model spaces subadditive}
For all $a \leq 0$, $\ipm{a}$ is {strictly} concave, and hence {strictly} subadditive.
\end{lemma}
\begin{proof}
Arguing exactly as in the proof of Lemma \ref{lem: Aubin in model spaces}, by Equation \eqref{eqn: isop profile and s_b}, we straightforwardly see that for $b < 0$
\begin{equation*}
\ipm{b}''(v) = - \dfrac{(n-1)(-b)^{\frac{n-1}{2}}}{\omega_{n-1}(\sinh(\sqrt{-b}R_b(v)))^{n+1}} < 0,
\end{equation*}

and, for $b = 0$
\begin{equation*}
\ipm{0}''(v) = - \dfrac{n-1}{\omega_{n-1}(R_0(v))^{n+1}} < 0.
\end{equation*}

The two last displayed equations implies that $\ipm{a}$ is concave for all $a \leq 0$. Fix $a \leq 0$ arbitrarily and define $f(v) := \ipm{a}(v)$. Since $f$ is strictly concave and $f(0)=0$, we obtain for any $v\neq 0$ that
\begin{equation}\label{eqn: isop profile inequality}
f(tv) = f(tv+(1-t)0) > tf(v), \text{ for any }t\in (0,1).
\end{equation}

Thus, for any $v,w\neq 0$, we get
\begin{align*}
f(v) + f(w) &= f\left((v+w)\frac{v}{v+w}\right) + f\left((v+w)\frac{w}{v+w}\right) \\
\overset{\eqref{eqn: isop profile inequality}}&{>} f(v+w)\frac{v}{v+w} + f(v+w)\frac{w}{v+w} \\
&= f(v+w).
\end{align*}

Since $a$ was taken arbitrarily, we conclude that $\ipm{a}$ is subadditive for all $a \leq 0$.
\end{proof}

\section{The conjecture for small normal graphs}

In Lemma~\ref{lem: isoperimetric inequality for normal graphs}, we show that the conjecture holds true for $C^1$ small normal graphs {over a $g$-geodesic ball} in the Euclidean space with a Cartan-Hadamard metric $g$. To the best of our knowledge, this was not known for arbitrary $\bar k$. For $\bar k =0$, a version of this was proved by Ghomi and Stavroulakis in \cite{ghomi2026local}.

Before stating the lemma, we need to set some notation. In polar coordinates, the metric of the model space of curvature
$\bar k$ is $dr^2+s_{\bar k}(r)^2\sigma$, where $\sigma$ is the round metric on $\mathbb{S}^{n-1}$. Let $(\R^n,g)$ be a Cartan-Hadamard manifold with $\underline k\leq \Sec_g \leq \bar k \leq 0$, $r = d_g(0,\cdot)$ and, in polar geodesic coordinates around $0$, we can write
\begin{equation}\label{eq: polar form of g}
    g=dr^2+s_{\bar k}(r)^2Q_r.
\end{equation}
A straightforward computation shows that $\partial_r (s_{\bar k}^2Q_r) = 2\mathrm{Hess}_gr$ on $T\bbS^{n-1}$\footnote{Let $\{\partial_{\theta_i}\}$ be a orthonormal frame for $T\mathbb{S}^{n-1}$, by metric compatibility of the Levi-Civita connection and its torsion-free property, we get $\partial_r g(\partial_{\theta_i},\partial_{\theta_j}) = g(\nabla_{\partial_r}\partial_{\theta_i}, \partial_{\theta_j}) + g(\partial_{\theta_i}, \nabla_{\partial_r}\partial_{\theta_j}) = g(\nabla_{\partial_{\theta_i}}\partial_r, \partial_{\theta_j}) + g(\partial_{\theta_i}, \nabla_{\partial_{\theta_j}}\partial_r) = \mathrm{Hess}_gr(\partial_{\theta_i},\partial_{\theta_j}) + \mathrm{Hess}_gr(\partial_{\theta_j},\partial_{\theta_i}) = 2\mathrm{Hess}_gr(\partial_{\theta_j},\partial_{\theta_i}).$ }. By $\Sec_g \leq \bar k\leq 0$ and the Hessian comparison {theorem} (cf. \cite[Chapter 6]{Petersen2016}), we have that $\partial_r (s_{\bar k}^2Q_r) \geq 2 s^\prime_{\bar k}s_{\bar k}Q_r$ and thus $\partial_r Q_r \geq 0$. It is easy to see that $Q_r\longrightarrow\sigma$ as $r\downarrow0$ and therefore
\begin{equation}\label{eq: Q greater than sigma}
    Q_r\geq\sigma.
\end{equation}
Similarly, from  Hessian comparison and $\Sec_g\geq\underline k$, we obtain $\partial_r (s_{\bar k}^2 Q_r) \leq 2s^\prime_{\underline k}s_{\bar k}^2Q_r/s_{\underline k}$ and this gives $\partial_r\bigl(s_{\bar k}^2s_{\underline k}^{-2}Q_r\bigr)\leq0$. Therefore, we conclude 
\begin{equation}\label{eq: Q smaller than model}
    Q_r\leq \frac{s_{\underline k}(r)^2}{s_{\bar k}(r)^2}\,\sigma.
\end{equation}

Define
\begin{equation}\label{eq: definition of q}
q(r,\theta):=\sqrt{\det\nolimits_\sigma Q_r(\theta)},
\end{equation}
equivalently $ d\vol_{Q_r} = q(r,\theta)d\vol_\sigma $.

\begin{lemma}\label{lem: isoperimetric inequality for normal graphs}
Let $(\R^n, g)$ be a Cartan-Hadamard manifold with $\underline k\leq\Sec_g\leq \bar k \leq 0$. For every $R_0 > 0$, there exist $\epsilon\in(0,\frac12]$ and $c>0$, depending only on $n,\underline k,\bar k, R_0$, with the following property. 

Let $0<R\leq R_0$ and let $E \subset \mathbb{R}^n$ be a domain such that $\partial E$ is given by the normal graph of $u \in C^1(\partial B_R^g,\mathbb{R})$ with respect to $g$ and $\|u/R\|_{L^\infty} + \|\nabla u\|_{L^\infty}\leq \epsilon$. Then
\begin{equation}\label{eq: quantitative normal graphs}
    P_{g}(E) - \ipm{\bar k}(\vol_g(E)) \geq c R^{n-1}\int_{\mathbb{S}^{n-1}}\bigl(q(R+u(\theta),\theta)-1\bigr)\,d\sigma \geq 0,
\end{equation}
Furthermore, if $P_{g}(E) = \ipm{\bar k}(\vol_g(E))$, then $g \simeq \delta^{\bar k}$ on $E$.
\end{lemma}
\begin{proof}
By \eqref{eq: Q greater than sigma}, we have $q(r,\theta)\geq1$, and $r\mapsto q(r,\theta)$ is nondecreasing. The eigenvalues of {$Q_r$} relative to $\sigma$ are {$\mu_i\geq1$} by \eqref{eq: Q greater than sigma}, {and then the eigenvalues of $\bigl(\det_\sigma Q_r\bigr)Q_r^{-1}$ are $\prod_{j\neq i}\mu_j\geq1$.}

Let $E_{\bar k}\subset \mathbb{M}^n_{\bar k}$ be the region with the same radial graph $\partial E = \{(R+u(\theta),\theta):\theta\in\mathbb{S}^{n-1}\}$ in the model space. Using the monotonicity of $q$, we obtain
\begin{equation}\label{eq: volume excess}
\begin{aligned}
    0
    &\leq
    \operatorname{vol}_g(E)
    -
    \vol_{\delta^{\bar k}}(E_{\bar k})
    =
    \int_{\mathbb{S}^{n-1}}\int_0^{R+u(\theta)}
    s_{\bar k}(r)^{n-1}
    \bigl(q(r,\theta)-1\bigr)\,dr\,d\sigma
    \\
    &\leq
    \int_{\mathbb{S}^{n-1}}\bigl(q(R+u(\theta),\theta)-1\bigr)\int_0^{R+u(\theta)}
    s_{\bar k}(r)^{n-1}\,dr\,d\sigma.
\end{aligned}
\end{equation}
The metric induced by \eqref{eq: polar form of g} on the graph
$r=R+u(\theta)$ is $s_{\bar k}(R+u)^2Q_{R+u}+du\otimes du$. The matrix determinant lemma therefore gives
\begin{equation*}
\begin{aligned}
    P_g(E)
    &= \int_{\mathbb{S}^{n-1}} \sqrt{\mathrm{det}(s_{\bar k}(R+u)^2Q_{R+u}+du\otimes du)} d\sigma\\
    &= \int_{\mathbb{S}^{n-1}}
    s_{\bar k}(R+u)^{n-1}q(R+u,\theta)
    \sqrt{
        1+
        s_{\bar k}(R+u)^{-2}
        Q_{R+u}^{-1}(du,du)
    }\,d\sigma.
    \label{eq: graph area}
\end{aligned}
\end{equation*}
Since $q(R+u,\theta)^2Q_{R+u}^{-1}(du,du)\geq|du|_\sigma^2$, it follows that
\[
    P_g(E)
    \geq
    \int_{\mathbb{S}^{n-1}}
    s_{\bar k}(R+u)^{n-1}
    \sqrt{
        q(R+u,\theta)^2
        +
        \frac{|du|_\sigma^2}{s_{\bar k}(R+u)^2}
    }\,d\sigma.
\]

On the other hand,
\[
    P_{\delta^{\bar k}}(E_{\bar k})
    =
    \int_{\mathbb{S}^{n-1}}
    s_{\bar k}(R+u)^{n-1}
    \sqrt{
        1+
        \frac{|du|_\sigma^2}{s_{\bar k}(R+u)^2}
    }\,d\sigma.
\]

For $z\geq0$ and $q\geq1$, we have that $\sqrt{q^2+z^2}-\sqrt{1+z^2} = \int_1^q\frac{t}{\sqrt{t^2+z^2}}\,dt \geq \frac{q-1}{\sqrt{1+z^2}}$. Therefore
\begin{equation}\label{eq: perimeter excess}
    P_g(E)
    -
    P_{\delta^{\bar k}}(E_{\bar k})
    \geq
    \int_{\mathbb{S}^{n-1}}
    s_{\bar k}(R+u)^{n-1}\,
    \frac{q(R+u(\theta),\theta)-1}{\sqrt{1+s_{\bar k}(R+u)^{-2}|du|_\sigma^2}}
    \,d\sigma.
\end{equation}

By Lemma \ref{lem: isop profile model spaces subadditive}, $\ipm{\bar k}$ is concave and increasing, and by \eqref{eqn: isop profile and s_b} and \eqref{eqn: isop profile and s_0}, we recall that $\ipm{\bar k}'(V_{\bar k}(r))=(n-1)s_{\bar k}'(r)/s_{\bar k}(r)$, which is nonincreasing in $r$. Since $\|u\|_{L^\infty}\leq\epsilon R$, we have $\vol_{\delta^{\bar k}}(E_{\bar k})\geq V_{\bar k}((1-\epsilon)R)$, and hence \eqref{eq: volume excess} gives
\begin{equation}\label{eq: profile excess}
\begin{aligned}
    &\ipm{\bar k}\bigl(\operatorname{vol}_g(E)\bigr)
    -
    \ipm{\bar k}\bigl(
        \vol_{\delta^{\bar k}}(E_{\bar k})
    \bigr)
    \\
    &\quad\leq
    (n-1)\frac{s_{\bar k}'((1-\epsilon)R)}{s_{\bar k}((1-\epsilon)R)}
   \int_{\mathbb{S}^{n-1}}\bigl(q(R+u(\theta),\theta)-1\bigr)\int_0^{R+u(\theta)}
    s_{\bar k}(r)^{n-1}\,dr\,d\sigma.
\end{aligned}
\end{equation}
Combining \eqref{eq: perimeter excess}{,} \eqref{eq: profile excess}, {and $P_{\delta^{\bar k}}(E_{\bar k})\geq\ipm{\bar k}(\vol_{\delta^{\bar k}}(E_{\bar k}))$,} we conclude that
{\begin{equation}\label{eq : perimeter excess}
    P_g(E)
    -
    \ipm{\bar k}\bigl(\operatorname{vol}_g(E)\bigr)
    \geq
    \int_{\mathbb{S}^{n-1}}\bigl(q(R+u(\theta),\theta)-1\bigr)\,\Gamma(\theta)\,d\sigma,
\end{equation}}
{where}
{\[
\Gamma(\theta):=\frac{s_{\bar k}(R+u)^{n-1}}{\sqrt{1+s_{\bar k}(R+u)^{-2}|du|_\sigma^2}}-(n-1)\frac{s_{\bar k}'((1-\epsilon)R)}{s_{\bar k}((1-\epsilon)R)}\int_0^{R+u(\theta)}s_{\bar k}(r)^{n-1}\,dr.
\]}

{It remains to show that $\Gamma\geq cR^{n-1}$. By $\|u\|_{L^\infty}\leq\epsilon R$, $(1-\epsilon)R\leq R+u\leq(1+\epsilon)R$. The metric induced by $g$ on $\partial B_R^g$ is $s_{\bar k}(R)^2Q_R$, by $\|u\|_{C^1}\leq\epsilon$, we get \eqref{eq: Q smaller than model} gives $|du|_\sigma\leq s_{\underline k}(R)|\nabla^{s_{\bar k}^2Q_r} u|_{s_{\bar k}^2Q_r}\leq\epsilon s_{\underline k}(R)$. Since $s_{\bar k}(r)\geq r$, $r\mapsto s_{\underline k}(r)/r$ is nondecreasing, and $R+u \geq (1-\epsilon)R \geq R/2$, we get $s_{\bar k}(R+u)^{-1}|du|_\sigma\leq 2\epsilon s_{\underline k}(R_0)/R_0=:\epsilon C_0$. 

Using the monotonicity of $s_{\bar k}$ and $s_{\bar k}(Rt)=Rs_{\bar kR^2}(t)$, $s'_{\bar k}(Rt)=s'_{\bar kR^2}(t)$, we obtain $\Gamma\geq R^{n-1}\gamma(\bar kR^2,\epsilon)$, where}
\[
\gamma(b,\epsilon):=\frac{s_b(1-\epsilon)^{n-1}}{\sqrt{1+C_0^2\epsilon^2}}-(n-1)\frac{s_b'(1-\epsilon)}{s_b(1-\epsilon)}\int_0^{1+\epsilon}s_b(t)^{n-1}\,dt.
\]
We claim that $\gamma(b,0)>0$ for every $b\leq0$. Indeed, write $s=s_b$ and let
\[
h(t):=s(t)^n-(n-1)s'(t)\int_0^ts(\tau)^{n-1}\,d\tau \text{ and } h(1)= s(1)\gamma(b,0).
\]
Using $s''=-bs$ and $(s')^2+bs^2=1$, we get $h'=s\varphi$ and $\varphi'=(n-2)s^{n-3}$, where $\varphi(t):=s(t)^{n-2}s'(t)+(n-1)b\int_0^ts(\tau)^{n-1}d\tau$. If $n\geq3$, then $\varphi(0)=0$, and if $n=2$, then $\varphi\equiv1$. In both cases $\varphi>0$ on $(0,\infty)$, hence $h>0$ on $(0,\infty)$ since $h(0)=0$, which proves the claim. Since $\gamma$ is continuous on $[\bar kR_0^2,0]\times[0,\frac12]$ and $\bar kR^2\in[\bar kR_0^2,0]$, by compactness we can choose $\epsilon\in(0,\frac12]$ and $c>0$, depending only on $n,\underline k,\bar k,R_0$, such that $\gamma(b,\epsilon)\geq c$ for every $b\in[\bar kR_0^2,0]$. Thus $\Gamma\geq cR^{n-1}$, and \eqref{eq : perimeter excess} gives \eqref{eq: quantitative normal graphs}.

For the rigidity statement, suppose that $P_g(E)=\ipm{\bar k}\bigl(\operatorname{vol}_g(E)\bigr)${. Then \eqref{eq: quantitative normal graphs} and the continuity of $q$ imply that} $q(R+u(\theta),\theta)=1$ for every $\theta${, and by the monotonicity of $q$ we get $q(r,\theta)=1$ for $0<r\leq R+u(\theta)$.} The eigenvalues of
{\(Q_r\)} relative to \(\sigma\) are at least one by \eqref{eq: Q greater than sigma}, while their product is {$q(r,\theta)^2=1$}. Thus every eigenvalue is equal to one and $Q_r=\sigma$ for {$0<r\leq R+u(\theta)$.} Hence
\begin{equation}\label{eq: isometry on E}
g=dr^2+s_{\bar k}(r)^2\sigma
\qquad\text{on }E.
\end{equation}
Moreover, by \eqref{eq: volume excess}, \eqref{eq: perimeter excess} {and \eqref{eq: isometry on E}}, we have $\vol_g(E) = \vol_{\delta^{\bar k}}(E_{\bar k})$ and $P_g(E) = P_{\delta^{\bar k}}(E_{\bar k}) = \ipm{\bar k}\bigl(\vol_{\delta^{\bar k}}(E_{\bar k})\bigr)$.  Therefore \(E_{\bar k}\) is an isoperimetric region in
\((\mathbb M_{\bar k}^n,\delta^{\bar k})\), and hence it is a geodesic ball. Consequently, by \eqref{eq: isometry on E}, \((E,g)\) is isometric to a geodesic ball in \((\mathbb M_{\bar k}^n,\delta^{\bar k})\), which proves the rigidity statement.
\end{proof}

\section{Main reduction}\label{sec:main reduction}

In this section, we address one of the main difficulties: the Cartan-Hadamard manifold may fail to admit isoperimetric regions, even under a lower Ricci curvature bound. We show that the problem can be reduced to studying an isoperimetric region in an Alexandrov space with two-sided curvature bounds. 

The result below is our main theorem (Theorem \ref{thm: main theorem}) in the more general context of Alexandrov spaces, cf. Remark \ref{CH - CAT}.

\begin{theorem}\label{thm: CH in alexandrov}
Let $(X, d)$ be a Alexandrov space with {no Alexandrov boundary}, curvature bounded from above by $\bar k \leq 0$ and below by $\underline k$, and noncollapsing constant $v_{nc}>0$ as in Definition \ref{def:alex}. There exists $v_0 = v_0(v_{nc}, \underline k,\bar k, n)>0$ such that $$\ip{X}{d}(v)\geq \ipm{\bar k}(v)\text{ for any }v\in (0,v_0].$$ 
Moreover, if there exists a finite perimeter set $E\subset X$ with $\cH^n_d(E) = v\in (0,v_0]$ and $P_d(E) = \ipm{\bar k}(v)$, then $E$ is isometric to a geodesic ball in $\mathbb{M}^n_{\bar k}$.
\end{theorem}
\begin{remark}\label{CH - CAT}
Theorem \ref{thm: main theorem} follows from Theorem \ref{thm: CH in alexandrov} since any Cartan-Hadamard manifold satisfies the global condition of being $\textup{CAT}(\bar k)$. Moreover, any pointed GH-limit of a Cartan-Hadamard manifold with $\underline k \leq \Sec_g \leq \bar k \leq 0$ is an Alexandrov space with no Alexandrov boundary and curvature bounded above by $\bar k$ and below by $\underline k$. Indeed, any pointed GH-limit of $(X,d, p_i)$ is locally geodesically complete by \cite[Example 4.3]{lytchak2019geodesically} and \cite[\S 2.2]{lytchak2019geodesically}. This and an application of \cite[Proposition 3.13(iii)]{kapovitch2022structure} conclude the proof.
\end{remark}

Fix a volume $v > 0$. Since $(X,d)$ is a $\textup{RCD}((n-1)\underline k, n)$ space by Proposition \ref{rmk:alex}\eqref{rmk: Alex RCD}, we are under the assumptions of \cite[Theorem 1.1]{antonelli2022isoperimetric} which gives a minimizing sequence $\Omega_i\subset X$,  for any $i\in\N$, that can be split into diverging and converging pieces ($\Omega_i^d$ and $\Omega_i^c$, respectively) as we give a precise description below. We have

\begin{itemize}
    \item $\Omega_i^c\to \Omega_0\subset X$ in the sense of sets of finite perimeter and $\Omega_0$ is an (possibly empty) isoperimetric region {for its own volume $0\le\mathcal{H}_d^n(\Omega_0)\le v$}.
\end{itemize}

There exist a positive natural number $
\bar N$ and points $p_{i,j}\in X$, for $j\in\{1,\ldots, \bar N\}$, such that $(X, d, \cH^n_d, p_{i,j}) \to (X_j,d_j,\cH^n_{d_j}, p_j)$ in the pmGH sense, where each space $(X_j,d_j,\cH^n_{d_j})$ is an $\textup{RCD}((n-1)\underline{k}, n)$ space which is also  non-collapsing, see \cite[Theorem 1.2]{de2018non}. We have

\begin{itemize}
    \item $\Omega_{i,j}^d \to \Omega_j^\infty\subset X_j$ in the sense of sets of finite perimeter, $P_d(\Omega_{i,j}^d)\to P_{d_j}(\Omega_j^\infty)$, and $\Omega_j^\infty$ is an isoperimetric region in $(X_j,d_j)$ for any $j\in\{1\ldots, \bar N\}$. 
\end{itemize}

Moreover, the isoperimetric profile and the volume split also as follows:

\begin{itemize}
    \item $\ip{X}{d}(v) = P_d(\Omega_0) + \sum_{j=1}^{\bar{N}}P_{d_j}(\Omega_j^\infty) = \ip{X}{d}(\cH^n_d(\Omega_0)) + \sum_{j=1}^{\bar{N}}\ip{X_j}{d_j}(\cH^n_{d_j}(\Omega_j^\infty))$, and
    
    \item $v = \cH^n_d(\Omega_0) + \sum_{j=1}^{\bar{N}}\cH^n_{d_j}(\Omega_j^\infty)$.
\end{itemize}

Next, since $(X,d)$ is $\textup{CAT}(\bar k)$, we apply \cite[Proposition 2.1.1]{alexander2019invitation} to derive that $(X_j,d_j)$ is $\textup{CAT}(\bar k)$ as well for each $j\in\{1,\ldots,\bar N\}$. Since each $(X_j,d_j,\cH^n_{d_j})$ is also an $\textup{RCD}((n-1)\underline k, n)$ space, we can apply \cite[Theorem 5.1(a)]{kapovitch2020cd} to obtain that each $(X_j, d_j)$ is an Alexandrov space with the same curvature bounds of $(X,d)$ {and the noncollapsing constants}.

We show how to conclude the proof of Theorem \ref{thm: CH in alexandrov} using the proposition below whose proof is done in Section~\ref{sec:small volumes}.

\begin{proposition}\label{prop: reduction assuming existence}
Let $(Z, d_Z)$ be an Alexandrov space with {no Alexandrov boundary}, curvature bounded from above by $\bar k\leq 0$ and below by $\underline k$, and noncollapsing constant $v_{nc}>0$ as in Definition \ref{def:alex}. Then there exists $v_0 = v_0(v_{nc}, \underline k, \bar k, n)>0$ such that, for every isoperimetric region $E\subset Z$ of volume $v\in (0,v_0]$, we have 
$$P_{d_Z}(E) \geq \ipm{\bar k}(v).$$
Moreover, if equality holds, then $E$ is isometric to a geodesic ball in $\mathbb{M}^n_{\bar k}$.
\end{proposition}
By the construction above and this proposition, we derive that
\begin{equation*} \ip{X}{d}(v) = P_d(\Omega_0) + \sum_{j=1}^{\bar{N}}P_{d_j}(\Omega_j^\infty) \geq \ipm{\bar k}(\cH^n_d(\Omega_0)) + \sum_{j=1}^{\bar{N}}\ipm{\bar k}(\cH^n_{d_j}(\Omega_j^\infty)). \end{equation*}

Using that $ \ipm{\bar k}$ is subadditive by Lemma \ref{lem: isop profile model spaces subadditive}, and 
$$v= \cH^n_d(\Omega_0) + \sum_{j=1}^{\bar{N}}\cH^n_{d_j}(\Omega_j^\infty),$$ 
we obtain $ \ip{X}{d}(v) \geq  \ipm{\bar k}(v)$ which concludes the proof of Theorem \ref{thm: CH in alexandrov} assuming the validity of Proposition \ref{prop: reduction assuming existence}. 

The rigidity of the equality case is already included in Proposition~\ref{prop: reduction assuming existence}.

\section{Proof for small volumes}\label{sec:small volumes}

Here we prove Proposition \ref{prop: reduction assuming existence} (which establishes Theorem \ref{thm: main theorem}, as explained in Section \ref{sec:main reduction}), under the small volume assumption.

\subsection{The inequality.}
Let $g_Z$ be the $C^{1,\alpha}$ Riemannian metric on $Z$ given by Proposition~\ref{rmk:alex}\eqref{rmk: Alex C1alpha}. We prove that there exists $v_0 = v_0({v_{nc}}, \underline k,\bar k, n)>0$ such that if $0<v\leq v_0$ then the proposition holds. Let $E$ be an isoperimetric region of volume $v$.

By \cite[Lemma 3.2]{nardulli2014isoperimetric}\footnote{The $C^4$-regularity assumption on the metric in \cite[Lemma 3.2]{nardulli2014isoperimetric} is solely due to an application of \cite[Theorem 1]{nardulli2018regularity}, which requires a $C^4$ metric to apply Allard's regularity theorem. This assumption can be substantially relaxed by using the extension of Allard's theorem to Alexandrov spaces established in \cite{agnoletto2025allard} and then applies in our setting; see also the discussion therein.}, up to choosing $v_0$ small enough depending on $v_{nc}, \underline k, \bar k, n$, there exist $p\in Z$ and $R=R(v)>0$ such that $\partial E$ is given by the normal graph of a $C^{2,\alpha}$ function $u$ defined on $\partial B_R^{g_Z}(p)$, and $\|u/R\|_{C^{2,\alpha}}\to0$ as $v\to0$. We also recall that, by \cite[Proposition 4.21]{antonelli2024asymptotic} (see also \cite[Lemma 4.9]{nardulli2020sharp}), there exist $v_1$ and $C$ depending on ${v_{nc}},\underline k, \bar k, n$ such that, if $v_0\leq v_1$, we obtain
\begin{equation}\label{eqn: small diam small vol}
    \diam_{g_Z}(E) \leq Cv^{1/n}.
\end{equation}

Since $B_{R/2}^{g_Z}(p)\subset E$, \eqref{eqn: small diam small vol} gives $R\leq Cv^{1/n}$. Up to choosing $v_0$ smaller, we have $R\leq1$. Let $\epsilon=\epsilon(n,\underline k,\bar k,1)$ be given by Lemma \ref{lem: isoperimetric inequality for normal graphs} with $R_0=1$.

Let $g_i$ be the $C^\infty$ approximants of Proposition \ref{rmk:alex}\eqref{rmk: Alex approximants} that converge to $g_Z$ in $C^{1,\alpha}$ and that satisfy $2 \underline k \leq \Sec_{g_i} \leq \bar k \leq 0$. Apply Lemma~\ref{lem: isoperimetric inequality for normal graphs} with $\tilde E_i := (\exp_p^{g_i})^{-1}(E)$ and $\tilde g_i = ((\exp^{g_i}_p)^{-1})_*g_Z$

$$ P_{g_Z}(E) = \lim_i P_{\tilde g_i}(\tilde E_i) \geq \lim_i\ipm{\bar k}(\vol_{\tilde g_i }(\tilde E_i))= \ipm{\bar k}(\vol_{g_Z}(E)),$$
where we used the $C^{1,\alpha}$ convergence of $g_i$ to $g_Z$ and the fact that the perimeter and volume only depend on the Riemannian metric in $C^0$ topology. This concludes the proof of the inequality.

\subsection{The equality case.} Since $g_i\to g_Z$ in the $C^{1,\alpha}$-topology, for $i$ large $\partial E$ is the normal graph over $\partial B_R^{g_i}(p)$ of a function $u_i$ with small $C^1$-norm. Apply Lemma~\ref{lem: isoperimetric inequality for normal graphs} to $\tilde g_i=((\exp^{g_i}_p)^{-1})_*g_i$ and $\tilde E_i:=(\exp^{g_i}_p)^{-1}(E)$, all constants are independent of $i$, to get
\begin{align*}
cR^{n-1}\int_{\mathbb{S}^{n-1}}\bigl(q_i(R+u_i(\theta),\theta)-1\bigr)\,d\sigma &\leq P_{\tilde g_i}(\tilde E_i)-\ipm{\bar k}(\vol_{\tilde g_i}(\tilde E_i))\\
&= P_{g_i}(E)-\ipm{\bar k}(\vol_{g_i}(E))\\
&\longrightarrow P_{g_Z}(E)-\ipm{\bar k}(\vol_{g_Z}(E))=0,
\end{align*}
where $q_i$ is given by \eqref{eq: definition of q} for $\tilde g_i$. By the monotonicity of $q_i$ and the argument in the proof of the rigidity statement of Lemma \ref{lem: isoperimetric inequality for normal graphs}, the maps $\exp_p^{g_i}:(\tilde E_i,dr^2+s_{\bar k}(r)^2\sigma)\to(E,g_i)$ are almost isometries. Since $g_i\to g_Z$ in the $C^{1,\alpha}$-topology, they converge in $C^1$ to $\exp_p^{g_Z}$, which is therefore an isometry from $(\tilde E,dr^2+s_{\bar k}(r)^2\sigma)$ onto $(E,g_Z)$. As in the proof of Lemma \ref{lem: isoperimetric inequality for normal graphs}, $\tilde E$ is then an isoperimetric region in $(\mathbb{M}^n_{\bar k},\delta^{\bar k})$, hence a geodesic ball, and $g_Z\simeq\delta^{\bar k}$ on $E$. This concludes the proof of Proposition \ref{prop: reduction assuming existence}.

\section{Reduction to the case with lower Ricci curvature bound}\label{sec: Ricci}

In this section, we show that it suffices to prove \Cref{thm: assuming equality sem Ricci lower bound} under the additional assumption $\mathrm{Ric}_g \geq (n-1)\underline k$. Namely, we show how Proposition~\ref{thm: assuming equality in CH} implies Theorem~\ref{thm: assuming equality sem Ricci lower bound}: the inequality in Subsection \ref{ssec: inequality} and the moreover part in Subsection \ref{ssec: equality case}. We then prove Proposition~\ref{thm: assuming equality in CH} in Section~\ref{sec:arbitrary volumes}.

\begin{proposition}\label{thm: assuming equality in CH}
Let $(M^n,g)$ be a Cartan-Hadamard manifold with $\Sec_g \leq \bar k \leq 0$ and $\mathrm{Ric}_g \geq (n-1)\underline k$. Assume in addition that Assumption \ref{assump:equality case} holds. Then we have
\begin{equation*}
    \ip{M^n}{g}(v)\geq \ipm{\bar k}(v)\text{ for any }v\geq 0.
\end{equation*}
Moreover, if there exists a finite perimeter set $E\subset M$ with $\vol_g(E) = v>0$ and $P_g(E) = \ipm{\bar k}(v)$, then $E$ is isometric to a geodesic ball in $\mathbb{M}^n_{\bar k}$.
\end{proposition}

\subsection{The inequality}\label{ssec: inequality} Assume the validity of \Cref{thm: assuming equality in CH}. Let $(M^n,g)$ be as in Theorem \ref{thm: assuming equality sem Ricci lower bound}, $\Omega$ be a bounded set in $M^n$, and choose a big enough $R>0$ and $p\in\Omega$ such that $\Omega \Subset B_R^g(p)$. We will show below that there exists a metric $\hat g$ on $M^n$ such that $g = \hat g$ on $B_R^g(p)$ and $(M^n, \hat g)$ is under the assumptions of \Cref{thm: assuming equality in CH}. Applying it and using $g=\hat g$ on $B_R^g(p) \Supset \Omega$ yield 
\begin{equation}\label{1}
P_g(\Omega) = P_{\hat g}(\Omega) \geq \ipm{\bar k}(\vol_{\hat g}(\Omega)) = \ipm{\bar k}(\vol_{g}(\Omega)).
\end{equation} 
Recalling that the isoperimetric profile defined in \eqref{eq:IsoperimetricProfileDfefinition} can be equivalently defined by taking the infimum over bounded smooth sets of the fixed volume by \cite[Theorem 1]{NardulliFlores2020} and applying \eqref{1} to a minimizing sequence $\Omega_j$ of volume $v>0$, we obtain 
$$\ip{M^n}{g}(v) = \lim_{j\to +\infty}P_g(\Omega_j) \geq \ipm{\bar k}(v).$$

We now prove the existence of such a metric $\hat g$ to conclude the reduction by basically adapting our argument in the proof of Proposition \ref{rmk:alex}\eqref{rmk: Alex approximants}. Indeed, we fix a function $f:[0,R+1]\to [0,+\infty)$ such that
\begin{displaymath}
f\in C^\infty([0,R+1]), \ f\equiv 0\text{ on } [0,R], \ f(R+1) = 1, \text{ and }f^{''} \geq \sqrt{-\bar k}f^{'} >0\text{ on }(R,R+1).
\end{displaymath}

Fix $r(\cdot):= d_g(p,\cdot)$ and define $\tilde g := (1 - f\circ r)^{-2} g$. Notice that, by the Hessian comparison theorem, we have $\mathrm{Hess}_g r \geq \sqrt{-\bar k}(g - dr\otimes dr)$ and therefore
\begin{equation}\label{eqn: hess}
\mathrm{Hess}_g (f\circ r) = f^{''}dr\otimes dr + f^{'}\mathrm{Hess}_g r \geq \sqrt{-\bar k}f^{'}g.
\end{equation}
By a direct computation, for any $q\in B_{R+1}^g(p)$ and any plane $\Pi := \mathrm{span}\{v,w\}\subset T_qM$, we obtain that
\begin{equation}\label{eqn: sec conformal}
\begin{aligned}
    \Sec_{\tilde g}(\Pi) &= (1-f\circ r)^2\Sec_g(\Pi) - (1-f\circ r)\bigl(\mathrm{Hess}_g(f\circ r)(v,v) \bigr.\\
    &\quad \bigl. + \mathrm{Hess}_g (f\circ r)(w,w) \bigr)- |\nabla^g (f\circ r)|_g^2.
\end{aligned}
\end{equation} 
Using \eqref{eqn: hess}, the fact that $|Dr|_g = 1$, and $\Sec_g\leq\bar k$, we obtain
\begin{displaymath}
\Sec_{\tilde g} \leq \bar k(1-f\circ r)^2 - 2\sqrt{-\bar k}(1-f\circ r)f^{'}\circ r - (f^{'}\circ r)^2 = -\bigl(\sqrt{-\bar k}(1-f\circ r)+f^{'}\circ r\bigr)^2 \leq \bar k,
\end{displaymath}
where the last inequality follows from $f^{'}\geq\sqrt{-\bar k}f$ on $[0,R+1]$. The latter is obtained by integrating $f^{''}\geq\sqrt{-\bar k}f^{'}$ from $R$ and using $f(R)=f^{'}(R)=0$.
The lower bound is a simple consequence of \eqref{eqn: sec conformal}. Indeed, set
$$C := \textup{sup}_{\overline{B_{R+1}^g(p)}}\left( |\Sec_g| + |\nabla^g (f\circ r)|_g^2, \|\mathrm{Hess}_g(f\circ r)\|_g\right), $$
by \eqref{eqn: sec conformal}, we readily obtain $\Sec_{\tilde g} \geq -C$. We also notice that, since $(1-f\circ r(x))^{-2}$ explodes to $+\infty$ as $x$ approaches $\partial B_{R+1}^g(p)$, $(B_{R+1}^g(p), \tilde g)$ is complete and therefore it's a Cartan-Hadamard manifold.

Define $s:[0,R+1)\to [0,+\infty)$ and $s(t):= \int_0^t\frac{1}{1-f(u)}du$, then $s^{'}(t) > 0$ for any $t\in [0,R+1)$. Finally, set $\rho:= s^{-1}$ and
$$\Phi:  M\to B_{R+1}^g(p) \text{ by }\Phi(\exp_p^g(tv)) = \exp_p^g(s^{-1}(t)v),$$
for any $v\in \{T_p M:|v|_g = 1\}$ and $t\geq 0$. It's easy to see that $\Phi$ is a diffeomorphism and $\Phi(x) = x$ for any $x\in B^g_R(p)$. Setting $\hat g := \Phi^* \tilde g$, we have that $(M^n, \hat g)$ is now under the assumptions of \Cref{thm: assuming equality in CH}. Applying it we conclude the proof of the inequality.

\subsection{Proof of Corollary \ref{cor: isop region bdd}}\label{ssec: bdd}

The corollary states that isoperimetric regions are bounded. Since we have already established the inequality \eqref{eqn: conjecture sem Ricci arbitrary vol} of \Cref{thm: assuming equality sem Ricci lower bound} (assuming \Cref{thm: assuming equality in CH} that we prove later), we can replicate the very same proof of \cite[Theorem 3]{nardulli2014generalized} by the third-named author. Indeed, the author only uses the Ricci lower bound to apply the isoperimetric inequality for small volumes, namely, the results \cite[Theorems 2.3 and 2.4]{nardulli2014generalized}. This can now be replaced by our \eqref{eqn: conjecture sem Ricci arbitrary vol} and the argument proceeds verbatim.

\subsection{The equality case}\label{ssec: equality case}

Assuming the validity of Proposition~\ref{thm: assuming equality in CH}, we prove the moreover part of Theorem \ref{thm: assuming equality sem Ricci lower bound}. Let $E$ be a finite perimeter set in $M^n$ with $\vol_g(E) = v >0$ and $P_g(E) = \ipm{\bar k}(v)$. By Corollary \ref{cor: isop region bdd}, $E$ is bounded. Then we apply the same argument above with $\Omega = E$ to obtain $\hat g$ under the assumptions of Proposition~\ref{thm: assuming equality in CH} and with $g = \hat g$ on $B_R^g(p) \Supset E$. By Assumption~\ref{assump:equality case} applied to $E$ in $(M^n, \hat g)$, we obtain that $(E,\hat g)$ is isometric to a geodesic ball in $\mathbb{M}^n_{\bar k}$. Since $\hat g = g$ on $B_R^g(p)$, we conclude the proof.

\section{On the arbitrary volume regime}\label{sec:arbitrary volumes}

We now prove Proposition~\ref{thm: assuming equality in CH}, which implies Theorem~\ref{thm: assuming equality sem Ricci lower bound} as explained in Section~\ref{sec: Ricci}, in the more general setting of Alexandrov spaces.

\begin{proposition}\label{thm: assuming equality}
Let $(Z, d_Z)$ be an Alexandrov space with {no Alexandrov boundary}, curvature bounded from above by $\bar k\leq 0$ and below by $\underline k$, and noncollapsing constant $v_{nc}>0$ as in Definition \ref{def:alex}. Suppose in addition that Assumption~\ref{assump:equality case} holds. Then
$$ \ip{Z}{d_Z} \geq \ipm{\bar k}(v),\text{ for any }v\geq 0.$$
Moreover, if there exists a finite perimeter set $E\subset Z$ with $\cH^n_{d_Z}(E) = v>0$ and $P_{d_Z}(E) = \ipm{\bar k}(v)$, then $E$ is isometric to a geodesic ball in $\mathbb{M}^n_{\bar k}$.
\end{proposition}
\begin{remark}\label{rmk: uso da assump 1}
We use Assumption~\ref{assump:equality case} only to apply Lemma~\ref{lem: isoperimetric inequality for normal graphs}. Our argument produces functions that have small $C^1$-norms and are defined on the boundary of an isoperimetric region $F_\infty$. The volume of such isoperimetric region $F_\infty$ is as in Assumption~\ref{assump:equality case}. Lemma~\ref{lem: isoperimetric inequality for normal graphs} requires the domain of the functions to be a geodesic sphere and this is what we ask in Assumption~\ref{assump:equality case}.
\end{remark}
\begin{proof}
Let $g_Z$ be the $C^{1,\alpha}$ Riemannian metric on $Z$ and $v_0>0$ be the small volume threshold from Proposition~\ref{prop: reduction assuming existence}, and define $S\subseteq \R_+:=[0,+\infty)$ as the set of volumes $v$ for which the conjecture holds true. By Proposition~\ref{prop: reduction assuming existence}, we have $[0,v_0]\subseteq S$. We will show that $S$ is both closed and open, which implies, by connectedness, $S = \R_+$.

The fact that $S$ is closed follows from the continuity of $\ip{Z}{g_Z}$ and $\ipm{\bar k}$, see \cite[Theorem 2]{NardulliFlores2020} for the continuity of $\ip{Z}{g_Z}$.

We now prove that $S$ is open by means of contradiction. Let $v_*\geq v_0$ be the supremum among all $s>0$ such that $[0,s]\subseteq S$, and assume that $v_* < +\infty$. Take $v_i \searrow v_*$ with $v_i\notin S$, i.e., $\ip{Z}{g_Z}(v_i) < \ipm{\bar k}(v_i)$ for every $i$. By \cite[Definition 1.1 and Theorem 1]{NardulliFlores2020}, $\ip{Z}{g_Z}$ is infimized over smooth bounded sets and then there are smooth bounded sets $E_i\subset Z$ of volume $v_i$ such that 
\begin{equation}\label{eqn: contradiction hyp reversed inequality}
P_{g_Z}(E_i) <\ipm{\bar k}(v_i)\text{ for each }i.
\end{equation}

By the definition of $v_*$ and the continuity of both isoperimetric profiles, we must have
\begin{equation}\label{eqn: isoperimetric equality case}
\ip{Z}{g_Z}(v_*) = \ipm{\bar k}(v_*).     
\end{equation}

Since $\sup_i P_{g_Z}(E_i) < +\infty$, $\sup_i \cH_{d_Z}^n(E_i) = \sup_i v_i < +\infty$, and, by Remark \ref{rmk:locally asymptotic bdd geometry}, $(Z,g_Z)$ is of $C^{1,\alpha}$-locally asymptotic bounded geometry, we may apply \cite[Theorem 1.2]{antonelli2022isoperimetric} (see also \cite[Theorem 1, pg 64]{MunozFloresnardulli2020}). This yields a nondecreasing, possibly unbounded, sequence $\{N_i\}\subset\mathbb{N}_{\geq 1}$, points $p_{i,j} \in Z$, with $1 \leq j \leq N_i$ for any $i$, and pairwise disjoint subsets $E_{i,j} \subset E_i$ such that
\begin{itemize}
    \item $\lim_i d_Z(p_{i,j},p_{i,l}) = + \infty$, for any $j \neq l < \bar N + 1$, where $\bar N := \lim_i N_i \in \mathbb{N} \cup \{+\infty\}$;
    \item for every $1 \leq j < \bar N + 1$, the sequence $(Z,g_Z,\mathcal{H}^n_{d_Z},p_{i,j})$ converges in the $C^{1,\alpha}$-topology to a $C^{1,\alpha}$ pointed Riemannian manifold $(Z_j,g_j,\mathcal{H}^n_{d_j},p_j)$ as $i \to + \infty$;
    \item for every $1 \leq j < \bar N + 1$, there exist sets of finite perimeter $F_j \subset Z_j$ such that $(E_{i,j},g_Z,p_{i,j}) \to (F_j,g_j,p_j)$ in the pointed $C^{1,\alpha}$-flat topology\footnote{Although the authors in \cite{MunozFloresnardulli2020} assumes $C^0$-locally asymptotic bounded geometry and therefore only obtains $C^0$-flat convergence, their proof readily gives $C^{1,\alpha}$-flat convergence if one assumes $C^{1,\alpha}$-locally asymptotic bounded geometry.}, cf. Definition \ref{def:flat convergence}, and there holds
    \begin{equation}\label{eqn: volume decomposition}
    v_* = \sum_{j=1}^{\bar N} \mathcal{H}^n_{d_j}(F_j)
    \end{equation}
    and
    \begin{equation}\label{eqn: perimeter decomposition}
    P_{g_j}(F_j) \leq\lim_i P_{g_Z}(E_{i,j}).
    \end{equation}
\end{itemize}

By definition of pointed $C^{1,\alpha}$-flat convergence, for each $1 \leq j \leq \bar N$ and any $R_j > 0$, there exists a domain $\Omega_{R_j}$  satisfying $F_j \subset B_{R_j}(p_j) \subset \Omega_{R_j} \subset Z_j$, an integer $\nu_{R_j} \in \N$, and $C^{2,\alpha}$ embeddings $\Psi_{i,R_j}: \Omega_{R_j} \to Z$ such that, for all $i \geq \nu_{R_j}$, we have that $E_{i,j} \subset B_{R_j}(p_{i,j}) \subset \Psi_{i,R_j}(\Omega_{R_j})$.

We now claim that the decomposition $\{E_{i,j}\}_j$ of $E_i$ only has one connected component at the limit. More precisely, the claim—which we prove at the end—is {the} following:

\medskip 

\textit{\underline{Claim 1:}} $\bar N = 1$.

Since, by claim 1, $\bar N = 1$, denote $(Z_{\infty},g_{\infty}) := (Z_1,g_1)$, $F_{\infty} := F_1$, $E_{i} := E_{i,1}$, and $p_i := p_{i,1}$. By \cite[Proposition 2.19, eq. (2.17)]{antonelli2022isoperimetric}, we derive that
\begin{equation*}
    \ip{Z}{g_Z}(v_*) \leq \ip{Z_\infty}{g_\infty}(v_*) \leq P_{g_{\infty}}(F_\infty) \overset{\eqref{eqn: perimeter decomposition}}{\leq}  \lim_{i}P_{g_Z}(E_i) = \ip{Z}{g_Z}(v_*).
\end{equation*}
By the equality above above, $F_\infty$ is an isoperimetric region of volume $v_*$ and, by \eqref{eqn: isoperimetric equality case}, we have
\begin{equation}\label{eqn: smoothness of F_infty}
    P_{g_\infty}(F_\infty) = \ipm{\bar k}(v_*).
\end{equation}

As explained in Section \ref{sec:main reduction}, $(Z_\infty,d_{g_\infty})$ is a non-collapsing Alexandrov space with no Alexandrov boundary and the same curvature bounds as $(Z,d_Z)$ and \cite[Proposition 2.19, eq. (2.17)]{antonelli2022isoperimetric} gives $\ip{Z_\infty}{g_\infty}\geq\ip{Z}{g_Z}$. We can then apply Assumption~\ref{assump:equality case}.

Therefore, by the definition of $v_*$, $\vol_{g_\infty}(F_\infty) = v_*$, \eqref{eqn: smoothness of F_infty}, and Assumption \ref{assump:equality case}, $F_{\infty}$ is smooth and hence we can apply \cite[Theorem 1]{nardulli2018regularity}\footnote{The assumption of $C^4$-regularity of the metric in \cite[Theorem 1]{nardulli2018regularity} arises solely from the use of Allard's regularity theorem. It can be substantially weakened using the extension of Allard's theorem to Alexandrov spaces established in \cite{agnoletto2025allard} and then applies in our setting; see also the discussion therein.} with $T_i = \Psi_{i,R_1}^{-1}(E_i), (M,g) =(Z_\infty,{\Psi_{i,R_1}^{-1}}_*g_Z)$, and $B = F_{\infty}$, to obtain that $\partial \Psi_{i,R_1}^{-1}(E_i)$ are normal $C^{2,\alpha}$-graphs of a function $u_i$ over $\partial F_{\infty}$. So, defining $\tilde u_i := u_i \circ \Psi_{i,R_1}^{-1}$, we have that $\partial E_i$ is a normal $C^{2,\alpha}$-graph of $\tilde u_i$ over $\partial \Psi_{i,R_1}(F_\infty)$. 

Using that $\partial E_i$ is a normal graph over $\partial \Psi_{i,R_1}(F_\infty)$ and, by Assumption \ref{assump:equality case}, $F_\infty$ is isometric to geodesic ball in $\mathbb{M}^n_{\bar k}$, we will show at the end the following claim.

Let $g_j$ be the $C^\infty$ approximants of Proposition \ref{rmk:alex}\eqref{rmk: Alex approximants} that converge to $g_Z$ in $C^{1,\alpha}$ and that satisfy $2 \underline k \leq \Sec_{g_j} \leq \bar k \leq 0$.

\medskip 

\textit{\underline{Claim 2:}} $\partial E_i$ is the normal graph of $w_{i,j}\in C^1(\partial B^{g_j}_{r_j}(p_i))$ with $\|w_{i,j}\|_{C^1} \to 0$ as $i\to\infty$ for $p_{i}\in Z$ and $r_j>0$ with $\sup_j r_j <+\infty$.

Let $\epsilon=\epsilon(n,\underline k,\bar k,R_0)$ be given by Lemma \ref{lem: isoperimetric inequality for normal graphs} with $R_0=\sup_j r_j <+\infty$; it does not depend on $i$ nor $j$. By Claim 2, for $i$ large enough, $w_{i,j}$ has small $C^1$-norm with $R=r_j$. Hence applying Lemma~\ref{lem: isoperimetric inequality for normal graphs} with $\tilde E_i^j := (\exp_{p_i}^{g_j})^{-1}(E_i)$ and $\tilde g_i^j = (\exp_{p_i}^{g_j})^{-1}_*g_Z$ to obtain
\begin{equation*}
\begin{aligned}
    P_{g_Z}(E_i) = P_{\tilde g_i}(\tilde E_i) = \lim_j P_{\tilde g_i^j}(\tilde E_i^j) &\geq \lim_j \ipm{\bar k}(\vol_{\tilde g_i^j}(\tilde E_i^j)) \\
    &\geq \ipm{\bar k}(\vol_{\tilde g_i}(\tilde E_i)) = \ipm{\bar k}(\vol_{g_Z}(E_i)),
\end{aligned}
\end{equation*}
where $\tilde E_i := (\exp_{p_i}^{g_Z})^{-1}(E_i)$ and $\tilde g_i = (\exp_{p_i}^{g_Z})^{-1}_*g_Z$ and we used the fact that perimeter and volume depend $C^0$ on the metric and $g_j\to g_Z$ in $C^{1,\alpha}$. This contradicts \eqref{eqn: contradiction hyp reversed inequality}. This contradiction shows that $S$ is open and then $S = \R_+$ which concludes the proof of the inequality.

\medskip

For the rigidity part, since we have already established $\ip{Z}{d_Z} \geq \ipm{\bar k}$ for \emph{every} volume, we are under the conditions of Assumption~\ref{assump:equality case} for any choice of volume $V\geq 0$ and therefore the rigidity statement follows. This completes the proof of the proposition up to Claims 1 and 2.

\vspace{1cm}

\textit{\underline{Proof of claim 1:}} Since the sets $E_{i,j} \subset E_i$ are pairwise disjoint and satisfy $\cup_{j=1}^{\bar N} E_{i,j} = E_i$ we have
\begin{equation}\label{eqn: sum perimeter F_j}
\sum_{j=1}^{\bar N} P_{g_j}(F_j) \overset{\eqref{eqn: perimeter decomposition}}{\leq} \sum_{j=1}^{\bar N} \lim_i P_{g_Z}(E_{i,j}) = \lim_i P_{g_Z}(E_i) = \ip{Z}{g_Z}(v_*) \overset{\eqref{eqn: isoperimetric equality case}}{=} \ipm{\bar k}(v_*).
\end{equation}
Furthermore, if $\bar N > 1$, then, by \eqref{eqn: volume decomposition}, we have $v_j < v_*$ for every $1 \leq j \leq \bar N$. Hence, by the definition of $v_*$ and \cite[Proposition 2.19, eq. (2.17)]{antonelli2022isoperimetric},
\begin{equation}\label{eqn: perimeter F_j small volume}
P_{g_j}(F_j) \geq \ipm{\bar k}(v_j).
\end{equation}

Moreover, by Lemma \ref{lem: isop profile model spaces subadditive}, we have
\begin{equation*}
\begin{aligned}
    \sum_{j=1}^{\bar N} \ipm{\bar k}(v_j) \geq \ipm{\bar k} \left(\sum_{j=1}^{\bar N} v_j\right) \overset{\eqref{eqn: volume decomposition}}&{=} \ipm{\bar k}(v_*) \\
    \overset{\eqref{eqn: sum perimeter F_j}}&{\geq} \sum_{j=1}^{\bar N} P_{g_j}(F_j) 
    \overset{\eqref{eqn: perimeter F_j small volume}}{\geq} \sum_{j=1}^{\bar N} \ipm{\bar k}(v_j).
\end{aligned}
\end{equation*}
Therefore,
\begin{equation*}
\sum_{j=1}^{\bar N} \ipm{\bar k}(v_j) = \ipm{\bar k} \left(\sum_{j=1}^{\bar N} v_j\right).
\end{equation*}

Since $v_j > 0$ for at least two numbers $j$ (otherwise $\bar N = 1$) and, by Lemma \ref{lem: isop profile model spaces subadditive}, the function $t \mapsto \ipm{\bar k}(t)$ is \emph{strictly} subadditive, the last displayed equality gives a contradiction and then $\bar N=1$ and this proves the claim.

\vspace{1cm}

\underline{\textit{Proof of claim 2:}}
To simplify notation, let $F_i := \Psi_{i,R_1}(F_\infty)$, $g:=g_j$, $g_i := {\Psi_{i,R_1}} _* g_\infty$, $r=r_j$, and $\mathcal{I}_i:({B_r^{\delta_{\bar k}}(0)},\delta_{\bar k})\longrightarrow ({F_i}, g_i)$ be an isometry, i.e., $\mathcal{I}_i^*g_i=\delta_{\bar k}$. Set $p_i:=\mathcal{I}_i(0)$ and notice that the fact that $\sup_j r_j <+\infty$ is clear from the set-up.

\medskip

\noindent$\rightarrow$ \textit{Step 1: \(\partial E_i\) is close
to \(\partial B_r^g(p_i)\) both in position and in tangent plane.} 

Up to choosing $R_1$ big enough, we have that
\begin{equation*}
\varepsilon_i
:=
\sup_{B_{2r}^g(p_i)}
\| g_i - g\|_{C^{1,\alpha}}
\longrightarrow 0.
\end{equation*}

Let $G_i:\partial F_i\longrightarrow\partial E_i$ be the normal-graph parametrization of \(\partial E_i\) over \(\partial F_i\), namely,
\begin{equation*}
G_i(x)
=
\exp_x^{g_i}
\left(
\tilde u_i(x)\nu_{\partial F_i}^{g_i}(x)
\right).
\end{equation*}

Since $\|\tilde u_i\|_{C^{2,\alpha}} \to 0$ as $i$ grows (by \cite{nardulli2018regularity}), it is easy to see that
\begin{equation*}
\eta_i
:=
\sup_{x\in\partial F_i}
\left[
\frac{1}{r}d_g\bigl(G_i(x),x\bigr)
+
\left|
\nu_{\partial E_i}^g\bigl(G_i(x)\bigr)
-
P_{x,G_i(x)}^g\nu_{\partial F_i}^g(x)
\right|_g
\right]
\longrightarrow 0,
\end{equation*}
where \(P_{x,y}^g:T_xM\to T_yM\) denotes \(g\)-parallel transport along the unique
\(g\)-geodesic joining \(x\) to \(y\) for any $x,y\in\partial F_i$.

For \(v\in T_{p_i}M\) with $|v|_g=1$ and $0\leq t\leq r$, define
\begin{equation*}
A_i(v):=\frac{v}{|v|_{g_i}}, \quad
\gamma_{i,v}(t)
:=
\exp_{p_i}^{g_i}\bigl(tA_i(v)\bigr),
\quad
\gamma_v(t)
:=
\exp_{p_i}^g(tv).
\end{equation*}

Since \(d\mathcal{I}_i(0)\) maps the \(\delta_{\bar k}\)-unit sphere onto \(\{v\in T_{p_i}M, |v|_{g_i}=1\}\), every vector \(A_i(v)\) is of the form $A_i(v)=d\mathcal{I}_i(0)[\theta]$ for some $|\theta|_{\delta_{\bar k}} = 1$. The radial curve {$t\longmapsto\mathcal{I}_i(\exp_0^{\delta_{\bar k}}(t\theta))$} is a \(g_i\)-geodesic, and hence {$\gamma_{i,v}(t)=\mathcal{I}_i(\exp_0^{\delta_{\bar k}}(t\theta))$}. This in turn implies that
\begin{equation*}
\partial F_i
=
\left\{
\gamma_{i,v}(r):v\in T_{p_i}M, |v|_g=1
\right\}.
\end{equation*}

We now compare \(\gamma_{i,v}\) with \(\gamma_v\). By the formula for the difference of the Levi-Civita connections, we obtain
\begin{equation}\label{eqn: D_i bounded}
\sup_{B_{2r}^g(p_i)}|\nabla^{g_i}-\nabla^g|_g
\leq C\varepsilon_i.
\end{equation}

Moreover,
\begin{equation}\label{eqn: initial velocity}
|A_i(v)-v|_g\leq C\varepsilon_i.
\end{equation}

Because \(\gamma_{i,v}\) is a \(g_i\)-geodesic, it satisfies
\begin{equation}\label{eqn: geodesic eqn}
\nabla_{\gamma_{i,v}'}^g\gamma_{i,v}'
=
-(\nabla^{g_i}-\nabla^g)\bigl(\gamma_{i,v}',\gamma_{i,v}'\bigr).
\end{equation}

Thus, when viewed in the metric \(g\), the curve \(\gamma_{i,v}\) has acceleration
bounded by \(C\varepsilon_i\) (by \eqref{eqn: D_i bounded} and \eqref{eqn: geodesic eqn}), while its initial velocity differs from that of
\(\gamma_v\) by at most \(C\varepsilon_i\) (by \eqref{eqn: initial velocity}).

\medskip

From \(\underline k \leq \Sec_g\leq \bar k\leq 0\), the continuous dependence for the geodesic equation, and the Jacobi field description of the differential of the
geodesic flow, we can derive the following uniform estimate
\begin{equation}\label{eqn: uniform in p_i}
    \sup_{\substack{\{v\in T_{p_i}M:|v|_g=1\}\\0\leq t\leq r}}
\left(
d_g\bigl(\gamma_{i,v}(t),\gamma_v(t)\bigr)
+
\left|
\gamma_{i,v}'(t)
-
P_{\gamma_v(t),\gamma_{i,v}(t)}^g\gamma_v'(t)
\right|_g
\right)
\leq C\varepsilon_i .  
\end{equation}

For the sake of readability and since the proof of \eqref{eqn: uniform in p_i} is very technical, we prove it at the end.

Here and below, vectors at nearby points are compared by \(g\)-parallel
transport along the connecting geodesic. In particular, notice that
\(\gamma_{i,v}(r)\in\partial F_i\),
\(\gamma_v(r)\in\partial B_r^g(p_i)\), and $d_g(\gamma_{i,v}(r),\gamma_v(r))\leq C\varepsilon_i$.

{Through the isometry \(I_i\), the curve \(\gamma_{i,v}\) corresponds to a radial geodesic in the constant-curvature space form \((\mathbb{R}^n,\delta_{\bar k})\). By the Gauss lemma, \(\gamma_{i,v}'(r)\) is the outward
\(g_i\)-unit normal to \(\partial F_i\).}

Consequently, for
\(W\in T_{\gamma_{i,v}(r)}\partial F_i\),
\[
\left|g\bigl(\gamma_{i,v}'(r),W\bigr)\right|
=
\left|(g-g_i)\bigl(\gamma_{i,v}'(r),W\bigr)\right|
\leq C\varepsilon_i|W|_g .
\]

This and \eqref{eqn: uniform in p_i} imply
\begin{equation}\label{eqn: normal to parallel transport}
\left|
\nu_{\partial F_i}^g\bigl(\gamma_{i,v}(r)\bigr)
-
P_{\gamma_v(r),\gamma_{i,v}(r)}^g
\nu_{\partial B_r^g(p_i)}^g\bigl(\gamma_v(r)\bigr)
\right|_g
\leq C\varepsilon_i .
\end{equation}

Since \(v\mapsto\gamma_{i,v}(r)\) and \(G_i\) are diffeomorphisms onto their
images, \(v\mapsto G_i(\gamma_{i,v}(r))\) parametrizes \(\partial E_i\).
The definition of \(\eta_i\), the triangle inequality, and \eqref{eqn: uniform in p_i}, ensure the position estimate, namely,
\begin{equation}
d_g\bigl(G_i(\gamma_{i,v}(r)),\gamma_v(r)\bigr)
\leq C(\varepsilon_i+\eta_i),
\label{eqn: position estimate}
\end{equation}

The triangle inequality, $\|\tilde u\|_{C^{2,\alpha}}\to 0$, \eqref{eqn: normal to parallel transport}, and the same Jacobi field estimates to control the change from parallel transport along the two short connecting geodesics, guarantee the tangent plane estimate, namely,
\begin{equation}
   \left|
\nu_{\partial E_i}^g\bigl(G_i(\gamma_{i,v}(r))\bigr)
-
P_{\gamma_v(r),G_i(\gamma_{i,v}(r))}^g
\nu_{\partial B_r^g(p_i)}^g\bigl(\gamma_v(r)\bigr)
\right|_g
\leq C(\varepsilon_i+\eta_i).
\label{eqn: tangent estimate} 
\end{equation}

\medskip

\noindent$\rightarrow$ \textit{Step 2: A diffeomorphism between $\partial E_i$ and $\partial B^g_r(p_i)$.}

Since $(Z,g)$ is Cartan-Hadamard, \(\exp_{p_i}^g:T_{p_i}M\to M\) is a global diffeomorphism and so is
\begin{equation*}
\begin{aligned}
\mathcal{T}_i:
\partial B_r^g(p_i)\times(-r,\infty)
&\longrightarrow
M\setminus\{p_i\},\\
\mathcal{T}_i(x,s)
&:=
\exp_{x}^g
\left(
s\nu_{\partial B_r^g(p_i)}^g(x)
\right)
\end{aligned}
\end{equation*}

In particular, \(\partial B_r^g(p_i)\) admits a tubular neighborhood of uniform width, say $r$, and, by Step 1 (cf. \eqref{eqn: position estimate}), \(\partial E_i\) is contained in this tubular neighborhood. Let
\begin{equation*}
\pi_i:
\mathcal{T}_i\left(
\partial B_r^g(p_i)\times\left(-\frac r2,\frac r2\right)
\right)
\longrightarrow
\partial B_r^g(p_i)
\end{equation*}
be the corresponding normal projection. In these coordinates, we have $\ker D\pi_i
=
\operatorname{span}\{\partial_s\}$, where \(\partial_s\) is the radial unit vector. By \eqref{eqn: tangent estimate} in Step 1, we derive
\begin{equation}\label{eqn: normals}
g\bigl(\nu_{\partial E_i}^g,\partial_s\bigr)
\geq
1-C(\varepsilon_i+\eta_i).
\end{equation}

Therefore, $\pi_i|_{\partial E_i}:
\partial E_i\longrightarrow\partial B_r^g(p_i)$ is a local diffeomorphism. It remains to show that this local diffeomorphism is globally one-to-one. Consider $\Theta_i:\{T_{p_i}M:|v|_g=1\}\longrightarrow \{T_{p_i}M:|v|_g=1\}$ defined by
\begin{equation*}
\Theta_i(v)
:=
\frac{1}{r}
\left(\exp_{p_i}^g\right)^{-1}
\left(
\pi_i\bigl(G_i\bigl(\gamma_{i,v}(r)\bigr)\bigr)
\right).
\end{equation*}

The position estimate \eqref{eqn: position estimate} implies that $
\sup_{v\in \{T_{p_i}M:|v|_g=1\}}|\Theta_i(v)-v|_g
\longrightarrow 0$. Therefore, for large \(i\), \(\Theta_i\) is homotopic to the identity through $H_i(t,v) := |V|/|V|_g$ for $V:= {(1-t)v+t\Theta_i(v)}$.

It follows that \(\Theta_i\), and hence \(\pi_i|_{\partial E_i}\), has degree one. Since
\(\pi_i|_{\partial E_i}\) is a local diffeomorphism between compact connected
spheres, it is a covering map. Its degree is one, so it has only one sheet and
is therefore a diffeomorphism.

\medskip
\noindent $\rightarrow$ \textit{Step 3: Identifying the function whose normal graph is $\partial E_i$.}

We can finally define $w_i:\partial B_r^g(p_i)\longrightarrow\R$ by
\begin{equation*}
\left(
\pi_i|_{\partial E_i}
\right)^{-1}(x)
=
\exp_x^g
\left(
w_i(x)\nu_{\partial B_r^g(p_i)}^g(x)
\right).
\end{equation*}

It is then clear that $\partial E_i$ is given by the normal graph of $w_i$ over $\partial B^g_r(p_i)$. By the position estimate \eqref{eqn: position estimate}, we readily get
\begin{equation*}
\|w_i\|_{C^0(\partial B_r^g(p_i))}
\leq
C(\varepsilon_i+\eta_i).
\end{equation*}

In the coordinates given by \(\mathcal{T}_i\), the metric has the form $\mathcal{T}_i^*g=ds^2+h_{i,s}$,
where \(h_{i,s}\) is the metric induced on the parallel hypersurface to $\partial B^g_r(p_i)$ at signed
distance \(s\). The unit normal to the graph of \(w_i\) is
\begin{equation*}
\frac{
\partial_s-\nabla^{h_{i,w_i}}w_i
}{
\sqrt{
1+
|\nabla^{h_{i,w_i}}w_i|_{h_{i,w_i}}^2
}
}.
\end{equation*}

We notice that $\underline k\leq \Sec_g\leq\bar k \leq0$,
\(|s|\leq r/2\), and classical Jacobi field estimates, imply that \(h_{i,s}\) and
\(h_{i,0}\) are uniformly equivalent. With this and since, by \eqref{eqn: normals}, this normal differs from \(\partial_s\) by at most
\(C(\varepsilon_i+\eta_i)\), we conclude that $$\|\nabla^{\partial B_r^g(p_i)}w_i\|_{C^0(\partial B_r^g(p_i))} \leq C(\varepsilon_i+\eta_i).$$ 

This finally finishes the proof of claim 2.

\vspace{1cm}

\underline{\textit{Proof of \eqref{eqn: uniform in p_i}:}} Fix $i$, $v\in T_{p_i}M$, $|v|_g=1$, and we will work along the reference geodesic $\gamma_v$.

\medskip

\noindent $\rightarrow$ \noindent \textit{Step 1: Everything in the same $g$-connection.} 

Since $(Z,g)$ is Cartan-Hadamard, $\exp_q^g$ is a global diffeomorphism for
every $q\in M$, so $\gamma_v$ has no conjugate points and the difference vector
\[
y(t):=\bigl(\exp_{\gamma_v(t)}^g\bigr)^{-1}\bigl(\gamma_{i,v}(t)\bigr)
\ \in\ T_{\gamma_v(t)}M, \qquad 0\le t\le r,
\]
is well defined, with $|y(t)|_g=d_g\bigl(\gamma_v(t),\gamma_{i,v}(t)\bigr)$.
Fix a $g$-parallel orthonormal frame $E_1(t),\dots,E_n(t)$ along $\gamma_v$
and identify $y(t)$ with its coordinate vector in $\mathbb R^n$; write
$\dot y,\ddot y$ for the covariant derivative of $y$ along $\gamma_v$, read
in this frame, so that covariant differentiation along $\gamma_v$ becomes
ordinary differentiation of coordinates. By definition $y(0)=0$,
 differentiating $\exp^g_{\gamma_v(t)}(y(t))=\gamma_{i,v}(t)$ at $t=0$,and \eqref{eqn: initial velocity}, give
\[
\dot y(0)=\gamma_{i,v}'(0)-\gamma_v'(0)=A_i(v)-v,
\qquad
|\dot y(0)|_g\leq C\varepsilon_i.
\]

\medskip

\noindent $\rightarrow$ \noindent \textit{Step 2: The ODE for $y$: a forced Jacobi equation.} 

Differentiate $\exp_{\gamma_v(t)}^g(y(t))=\gamma_{i,v}(t)$ twice in $t$. For the sake of readability, for $a\in\{1,\ldots, n\}$, set
\[
H(t,z):=\exp_{\gamma_v(t)}^g\bigl(z^aE_a(t)\bigr), \qquad
V_0:=H_*\partial_t = dH(\partial_t), \qquad V_a:=H_*\partial_{z^a} = dH(\partial_{z^a}),
\]
so that $\gamma_{i,v}(t)=H(t,y(t))$. Since $[\partial_t,\partial_{z^a}]=0$
and $\nabla^g$ is torsion-free, for $z=y(t)$, we get
\begin{equation}\label{eqn: forced jacob 0} 
\nabla_{\gamma_{i,v}'}^g\gamma_{i,v}'
=
\nabla^g_{V_0}V_0 +2\dot y^a\nabla_{V_0}^gV_a+\dot y^a\dot y^b\nabla_{V_a}^gV_b+\ddot y^aV_a,
\end{equation}

At $z=0$, since the differential
of $\exp_{\gamma_v(t)}^g$ at the origin is the identity,  we have $V_a(t,0)=E_a(t)$. It also follows that 
$\nabla_{V_0}^gV_0(t,0)=\nabla_{\gamma_v'}^g\gamma_v'=0$. Because $E_a$ is
parallel, $\nabla_{V_0}^gV_a(t,0)=\nabla_{\partial_t}E_a(t)=0$. This gives that 
\[
\nabla_{\gamma_{i,v}'}^g\gamma_{i,v}' =\dot y^a\dot y^b\nabla_{V_a}^gV_b+\ddot y^aV_a \ \text{ at } \ z=0.
\]

To isolate the coefficient of \(\ddot y\), define
\[
A_t(z):\mathbb R^n\longrightarrow T_{H(t,z)}M,
\qquad
A_t(z)e_a:=V_a(t,z).
\]

Since \((Z,g)\) is Cartan-Hadamard, the exponential map is a global diffeomorphism and thus \(A_t(z)\) is invertible. 

Define $B(t,z,w):=\nabla^g_{V_0}V_0(t,z)+2w^a\nabla^g_{V_0}V_a(t,z)+w^aw^b\nabla^g_{V_a}V_b(t,z)$ and $G(t,z,w):=A_t(z)^{-1}B(t,z,w)$. Since $\ddot y^aV_a=A_t(y)\ddot y$, equation \eqref{eqn: forced jacob 0} can be rewritten as
\begin{equation}\label{eqn: forced jacobi 1}
\nabla_{\gamma_{i,v}'}^g\gamma_{i,v}'
=
A_t(y)\bigl(\ddot y+G(t,y,\dot y)\bigr)\ \text{ and } \
\ddot y+G(t,y,\dot y)
=
A_t(y)^{-1}\nabla^g_{\gamma_{i,v}^\prime}\gamma_{i,v}^\prime.
\end{equation}

We now compute the linearization of \(G\) at \((z,w)=(0,0)\). First, we have that $B(t,0,0) = \nabla^g_{V_0}V_0(t,0) = 0,$ and hence $G(t,0,0)=0$.

The derivative with respect to $w$ readily gives
\[
\partial_{w^c}G(t,z,w)
=
A_t(z)^{-1}
\left(
2\nabla^g_{V_0}V_c(t,z)
+2w^b\nabla_{V_c}^gV_b(t,z)
\right),
\]
and, at \((z,w)=(0,0)\), since \(V_c(t,0)=E_c(t)\) and \(E_c\) is parallel, we derive
\begin{equation}\label{eqn: DwG at 0}
D_wG(t,0,0)e_c
=
2A_t(0)^{-1}\nabla^g_{V_0}V_c(t,0)
=
0.
\end{equation}

We must also compute the derivative with respect to $z$. We see that $D_zG[e_c]=D_z(A^{-1})[e_c]\,B+A^{-1}D_zB[e_c]$ and, at \((z,w)=(0,0)\), the first term vanishes because \(B(t,0,0)=0\). Thus
\begin{equation}\label{eqn: Dgz}
D_zG(t,0,0)e_c
=
A_t(0)^{-1}
\bigl[\nabla_{V_c}^g\nabla^g_{V_0}V_0\bigr]_{(t,0)}.
\end{equation}

By the curvature commutation identity, we have $\nabla_{V_c}^g\nabla^g_{V_0}V_0-\nabla^g_{V_0}\nabla_{V_c}^gV_0=R^g(V_c,V_0)V_0$ and, since \(\nabla_{V_c}^gV_0=\nabla^g_{V_0}V_c\), we get
\[
\nabla_{V_c}^g\nabla^g_{V_0}V_0
=
\nabla^g_{V_0}\nabla^g_{V_0}V_c+R^g(V_c,V_0)V_0.
\]

At \(z=0\), $V_c(t,0)=E_c(t)$, $V_0(t,0)=\gamma'(t)$, and since \(E_c\) is parallel, we see that $\nabla^g_{V_0}\nabla^g_{V_0}V_c(t,0)=0$ and thus $\nabla_{V_c}^g\nabla^g_{V_0}V_0(t,0) = R^g\bigl(E_c(t),\gamma'(t)\bigr)\gamma'(t)$. Plugging this into \eqref{eqn: Dgz}, we get
\[
D_zG(t,0,0)e_c
=
A_t(0)^{-1}
R^g\bigl(E_c(t),\gamma'(t)\bigr)\gamma'(t).
\]

Define $\mathcal R(t):\mathbb R^n\longrightarrow\mathbb R^n$ by $\mathcal R(t)u:=A_t(0)^{-1}\left(R^g(A_t(0)u,\gamma'(t))\gamma'(t)\right)$. So, by the last displayed equation, we have 
\begin{equation}\label{eqn: DzG at 0}
D_zG(t,0,0)=\mathcal R(t).
\end{equation}

By Taylor expansion, \eqref{eqn: DwG at 0}, and \eqref{eqn: DzG at 0}, we then derive
\[
G(t,z,w)
=
\mathcal R(t)z-\mathcal N(t,z,w),
\]
where $\mathcal N(t,0,0)=0$ and $D_{(z,w)}\mathcal N(t,0,0)=0$. Define $\mathcal F_i(t) := A_t(y(t))^{-1}\nabla^g_{\gamma_{i,v}'}\gamma_{i,v}'$. By the last displayed equation and \eqref{eqn: forced jacobi 1}, we finally derive the ODE—the forced Jacobi equation—for $y$: 
\begin{equation}
\label{eqn:forced-jacobi-2}
\ddot y(t)+\mathcal R(t)y(t)
=
\mathcal N\bigl(t,y(t),\dot y(t)\bigr)
+\mathcal F_i(t),
\qquad 0\le t\le r.
\end{equation}

It remains to estimate the two terms on the right-hand side of
\eqref{eqn:forced-jacobi-2}. Since \(\mathcal N\) and its first derivative vanish
at \((z,w)=(0,0)\), Taylor's formula with integral remainder gives, with 
\(\xi=(z,w)\in \overline{B_1}\times \overline{B_1}\) and $t\in[0,r]$,
\[
{
|\mathcal N(t,z,w)|
\le
C\bigl(|z|+|w|\bigr).
}
\]

Finally, since \(A_t(y(t))^{-1}\) is uniformly bounded on $\overline{B_1}$, by \eqref{eqn: D_i bounded}, we get
\begin{equation}\label{eqn: bound cal Fi}
\sup_{0\le t\le r}|\mathcal F_i(t)|_g = \sup_{0\le t\le r}\left|A_t(y(t))^{-1}\left(\nabla^g_{\gamma_{i,v}'} - \nabla^{g_i}_{\gamma_{i,v}'}\right)\gamma_{i,v}'\right|_g
\le C\varepsilon_i.
\end{equation}

\medskip

\noindent $\rightarrow$ \noindent \textit{Step 3: A first order ODE.} 

Set $Z(t):=(y(t),\dot y(t))\in\mathbb R^{2n}$, so that
\eqref{eqn:forced-jacobi-2} becomes the first-order linear system
\begin{equation}\label{eqn: transition matrix 715}
\dot Z(t)=\mathcal A(t)Z(t)+\bigl(0,\ \mathcal N(t)+\mathcal F_i(t)\bigr),
\;\;
\mathcal A(t)=\begin{pmatrix}0&I_n\\-\mathcal R(t)&0\end{pmatrix},
\;\;
\|\mathcal A(t)\|\leq1+\rho .
\end{equation}

This is the variational equation of
\cite[Chapter V, Theorem 3.1]{hartman2002ordinary}. 
Let $\Phi(t,s)$ be the state-transition matrix of the homogeneous system $\dot Z=\mathcal A(t)Z$, i.e., $\partial_t\Phi(t,s)=\mathcal A(t)\Phi(t,s), \Phi(s,s)=I$; a column of $\Phi(t,0)$ is $(J(t),\dot J(t))$ for
the Jacobi field $J$ along $\gamma_v$ with the corresponding initial data. So, bounding $\|\Phi\|$ amounts to bounding Jacobi fields on $[0,r]$. Since $\Sec_g\leq 0$, for any
Jacobi field $J$ along the unit-speed geodesic $\gamma_v$,
\[
\frac{d^2}{dt^2}|J|^2
=2|\dot J|^2-2\Sec_g(J,\gamma_v')\bigl(|J|^2-\langle J,\gamma_v'\rangle^2\bigr)
\ \geq\ 2|\dot J|^2\ \geq\ 0;
\]
hence $|J|^2$ is convex along $\gamma_v$,
$\gamma_v$ has no conjugate points, and
$\Phi$ cannot blow up on $[0,r]$. On the other hand, by Rauch comparison, we get
\begin{equation*}
\begin{aligned}
|J(t)|\leq\cosh(\sqrt{-\underline k}t)|J(0)|+\frac{\sinh(\sqrt{-\underline k}t)}{\sqrt{-\underline k}}|\dot J(0)|,\\
|\dot J(t)|\leq \sqrt{-\underline k}\sinh(\sqrt{-\underline k}t)|J(0)|+\cosh(\sqrt{-\underline k}t)|\dot J(0)| .    
\end{aligned}
\end{equation*}

Together, these two bounds give a constant
$K=K(n,r,\underline k,\bar k)<\infty$, depending only on $n$, $r$,
$\underline k$, $\bar k$, in particular not on $i$ or on $v$, such
that
\[
\|\Phi(t,s)\|\leq K \qquad\text{for all } 0\leq s\leq t\leq r .
\]

\medskip

\noindent $\rightarrow$ \noindent \textit{Step 4: Duhamel's Principle and Gr\"onwall's Inequality.}

Variation of constants on \eqref{eqn: transition matrix 715}, with
$Z(0)=(0,\,A_i(v)-v)$, gives
\[
Z(t)=\Phi(t,0)Z(0)+\int_0^t\Phi(t,s)\bigl(0,\ \mathcal N(s)+\mathcal F_i(s)\bigr)\,ds .
\]

As long as $|Z(s)|\leq1$ for $0\leq s\leq t$, taking norms and using
$\|\Phi\|\leq K$, $|Z(0)|\leq C\varepsilon_i$ (by
\eqref{eqn: initial velocity}), $|\mathcal F_i(s)|\leq C\varepsilon_i$ (by \eqref{eqn: bound cal Fi}), and
$|\mathcal N(s)|\leq C|Z(s)|$ gives
\[
|Z(t)|\leq KC(1+r)\varepsilon_i+KC\int_0^t|Z(s)|\,ds ,
\]
and Gr\"onwall's inequality yields
\begin{equation}\label{eqn: Gronwall}
|Z(t)|\leq KC(1+r)e^{KCr}\varepsilon_i=:C'\varepsilon_i,
\qquad 0\leq t\leq r,
\end{equation}
on the interval where $|Z|\leq1$ holds. Since $C'$
does not depend on $i$, once $i$ is large enough that $C'\varepsilon_i<1$
we have $|Z|\leq1$ on $[0,r]$; so the bound above holds on all of $[0,r]$ for
every large $i$.

It remains to translate the bound on $Z=(y,\dot y)$ into the two terms of
\eqref{eqn: uniform in p_i}. By definition $|y(t)|_g=d_g(\gamma_{i,v}(t),\gamma_v(t))$,
which is the first term.

For the second term, fix $t\in[0,r]$ and introduce the connecting variation
\[
  \sigma(s,t)=\exp^{g}_{\gamma_v(t)}\!\big(s\,y(t)\big),\qquad 0\leq s \leq 1,
\]
so that $\sigma(0,t)=\gamma_v(t)$, $\sigma(1,t)=\gamma_{i,v}(t)$, and, for each fixed
$t$, the curve $s\mapsto\sigma(s,t)$ is the unique minimizing $g$-geodesic joining
$\gamma_v(t)$ to $\gamma_{i,v}(t)$ of
length $|y(t)|_g$. Put $S:=\partial_s\sigma$ and $J:=\partial_t\sigma$. Since every
$s$-curve is a geodesic, $S$ is parallel along it, so $|S|_g\equiv|y(t)|_g$, and $J$
is a Jacobi field along it with $J(0)=\gamma_v'(t)$, $J(1)=\gamma_{i,v}'(t)$, and $\nabla_s J(0)=\dot y(t)$.

Let $P_s$ be $g$-parallel transport along $s\mapsto\sigma(s,t)$ from $0$ to $s$, so
$P_1=P^{g}_{\gamma_v(t),\gamma_{i,v}(t)}$ is the operator in
\eqref{eqn: uniform in p_i}, and set
$\widetilde J:=P_s^{-1}J\in T_{\gamma_v(t)}M$. The Jacobi equation becomes 
\[
  \widetilde J''(s)=-P_s^{-1}R\big(J,S\big)S ,
\]
and $\underline k\le\Sec_g\le\bar k\leq 0$, together with $|S|_g=|y(t)|_g$
and the Rauch bound $|J(s)|_g\le C$, gives
$|\widetilde J''(s)|_g\le C\rho\,|y(t)|_g^{2}$. Taylor's formula with integral
remainder across $s\in[0,1]$,
\[
  \widetilde J(1)=\widetilde J(0)+\widetilde J'(0)+E,\qquad
  |E|_g\le C\rho\,|y(t)|_g^{2},
\]
reads $P_1^{-1}\gamma_{i,v}'(t)=\gamma_v'(t)+\dot y(t)+E$; applying the isometry
$P_1$, we get
\[
  \gamma_{i,v}'(t)-P^{g}_{\gamma_v(t),\gamma_{i,v}(t)}\gamma_v'(t)
  =P^{g}_{\gamma_v(t),\gamma_{i,v}(t)}(\dot y(t)+  E).
\]

Since $P^{g}_{\gamma_v(t),\gamma_{i,v}(t)}$ is a linear $g$-isometry the right-hand side has $g$-norm $|\dot y(t)+E|_g$. By the triangle
inequality and the remainder bound $|E|_g\le C\rho\,|y(t)|_g^{2}$,
\[
  \bigl|\gamma_{i,v}'(t)-P^{g}_{\gamma_v(t),\gamma_{i,v}(t)}\gamma_v'(t)\bigr|_g
  \;=\;|\dot y(t)+E|_g
  \;\le\;|\dot y(t)|_g+C\rho\,|y(t)|_g^{2}.
\]

The Gr\"onwall bound (cf. \eqref{eqn: Gronwall}) gives $|y(t)|_g\le|Z(t)|\le C\varepsilon_i$ and
$|\dot y(t)|_g\le|Z(t)|\le C\varepsilon_i$; in particular $|y(t)|_g\le1$ for $i$
large, so $|y(t)|_g^{2}\le|y(t)|_g\le C\varepsilon_i$. Hence
\[
  \bigl|\gamma_{i,v}'(t)-P^{g}_{\gamma_v(t),\gamma_{i,v}(t)}\gamma_v'(t)\bigr|_g
  \;\le\;C\varepsilon_i,
\]
and, adding the first term $d_g(\gamma_{i,v}(t),\gamma_v(t))=|y(t)|_g\le C\varepsilon_i$, we obtain
\[
  d_g\bigl(\gamma_{i,v}(t),\gamma_v(t)\bigr)
  +
  \bigl|\gamma_{i,v}'(t)-P^{g}_{\gamma_v(t),\gamma_{i,v}(t)}\gamma_v'(t)\bigr|_g
  \ \le\ C\varepsilon_i, \qquad 0\le t\le r .
\]

\medskip

\noindent $\rightarrow$ \noindent \textit{Step 5: Uniformity in $v$ and $i$.}

Every constant occurring above $\rho$, $a$, $K$, $C$, $C'$ depends
only on $n$, $r$, and the curvature bounds
$\underline k\leq\Sec_g\leq\bar k\leq0$ on the fixed ball $B_{2r}^g(p_i)$.
None of them depends on the direction $v$, which enters only through the
initial datum $A_i(v)-v$, controlled uniformly on the compact sphere
$\{v\in T_{p_i}M:|v|_g=1\}$ by \eqref{eqn: initial velocity}; and, once $i$
is large enough for the bootstrap step above to apply, none of them depends
on $i$ either. Thus the same constant $C$ works for every unit
vector $v\in T_{p_i}M$ and every $t\in[0,r]$, and taking the supremum gives \eqref{eqn: uniform in p_i}. This finishes the proof.
\end{proof}

\bibliographystyle{abbrv}
\bibliography{9biblio}

\end{document}